\documentclass[11pt]{amsart}
\newcommand\version{public}

\usepackage{fullpage,setspace}
\usepackage{enumerate}   %
\usepackage{url}
\usepackage{enumerate}
\usepackage{amsmath,amssymb,amscd,mathrsfs,latexsym}
\usepackage{mathdots}
\newcommand{\choosefont}[1]{\usepackage{#1}}  %

\usepackage{ifthen}     %

\newcommand\pubpri[2]{%
\ifthenelse{\equal{\version}{public}}%
{{#1}}%
{\marginpar{\scshape\small Pubpri Alert}{#2}}}
\newcommand\pubprinoalert[2]{%
\ifthenelse{\equal{\version}{public}}%
{{#1}}%
{#2}}
\newcommand\ignore[1]{}
\usepackage{color}
\providecommand\wantcolor{yes}   %
\definecolor{backgroundyellow}{cmyk}{.2,.1,.8,.2}
\definecolor{backgroundblue}{rgb}{0,0,1}
\definecolor{backgroundred}{rgb}{1,0,0}
\definecolor{backgroundmagenta}{cmyk}{0,1,0,0}
\ifthenelse{\equal{\wantcolor}{no}}{\renewcommand\color[1]{}}{}

\newcommand\mysubsubsection[1]{%
		\subsubsection{\sffamily\upshape\mdseries #1}}

\newcommand\mysss{\mysubsubsection}

\usepackage{chngcntr}
\counterwithin{figure}{section}
\providecommand{\theoremnumbering}{section}

\ifthenelse{\equal{\theoremnumbering}{section}}{\newtheorem{annotation}{Annotation}[section]}%
{\ifthenelse{\equal{\theoremnumbering}{subsection}}{\newtheorem{annotation}{Annotation}[subsection]}{\newtheorem{annotation}{Annotation}}}
\newtheorem{theorem}[annotation]{%
		Theorem}
\newtheorem{lemma}[annotation]{%
		Lemma}
\newtheorem{definition}[annotation]{%
		Definition}
\newtheorem{corollary}[annotation]{%
		Corollary}
\newtheorem{proposition}[annotation]{%
		Proposition}
\newtheorem{conjecture}[annotation]{%
		Conjecture}
\newtheorem{example}[annotation]{%
		Example}
\newcommand\bexample{\begin{example}\begin{rm}}
\newcommand\eexample{\end{rm}\hfill$\Box$\end{example}}
\newtheorem{examplenobox}[annotation]{%
		Example}
\newcommand\bexamplenobox{\begin{examplenobox}\begin{rm}}
\newcommand\eexamplenobox{\end{rm}\end{examplenobox}}
\newtheorem{exercise}[annotation]{%
		Exercise}
\newcommand\bexercise{\begin{exercise}\begin{rm}}
\newcommand\eexercise{\end{rm}\end{exercise}}
\newtheorem{notation}[annotation]{%
		Notation}
\newcommand\bnotation{\begin{notation}\begin{rm}}
\newcommand\enotation{\end{rm}\end{notation}}

\newtheorem{remark}[annotation]{%
		Remark}
\newcommand\bremark{\begin{remark}%
\begin{upshape}}
\newcommand\eremark{\end{upshape}%
\end{remark}}
\newenvironment{remark*}{%
\par\noindent{\scshape 
  Remark: }\begin{rm}}{\hfill\end{rm}\newline} 
\newcommand\bremarkstar{\begin{remark*}}
\newcommand\eremarkstar{\end{remark*}}
\newcommand\bdefn{\begin{definition}%
\begin{upshape}}
\newcommand\edefn{\end{upshape}%
\end{definition}}
\newtheorem{caveat}[annotation]{%
		Caveat}
\newcommand\bcaveat{\begin{caveat}%
\begin{upshape}}
\newcommand\ecaveat{\end{upshape}%
\end{caveat}}
\newenvironment{caveatstar}{%
\par\noindent{\scshape\bfseries
  Caveat: }\begin{rm}}{\end{rm}\newline} 
\newcommand\bcaveatstar{\begin{caveatstar}}%
\newcommand\ecaveatstar{\end{caveatstar}}
\newenvironment{myproof}{%
\par\noindent{\scshape 
  Proof: }\begin{rm}}{\hfill$\Box$\end{rm}\newline} 
\newcommand\bmyproof{\begin{myproof}}
\newcommand\emyproof{\end{myproof}}
\newenvironment{myproofnobox}{%
\par\noindent{\scshape Proof: }\begin{rm}}{\end{rm}\hfill\newline}
\newcommand\bmyproofnobox{\begin{myproofnobox}}
\newcommand\emyproofnobox{\end{myproofnobox}}
\newenvironment{solution}{%
\par\noindent{\scshape Solution: }\begin{rm}}{\hfill$\Box$\end{rm}\newline}
\newenvironment{solutionnobox}{%
\par\noindent{\scshape Solution: }\begin{rm}}{\end{rm}}
\newcommand\bsolution{\begin{solution}\begin{rm}}
\newcommand\esolution{\end{rm}\end{solution}}
\newcommand\bsolutionnobox{\begin{solutionnobox}\begin{rm}}
\newcommand\esolutionnobox{\end{rm}\end{solutionnobox}}
\newcommand\bthm{\begin{theorem}}
\newcommand\ethm{\end{theorem}}
\newcommand\bcor{\begin{corollary}}
\newcommand\ecor{\end{corollary}}
\newcommand\blemma{\begin{lemma}}
\newcommand\elemma{\end{lemma}}
\newcommand\bprop{\begin{proposition}}
\newcommand\eprop{\end{proposition}}
\newcommand\beqn{\begin{equation}}
\newcommand\eeqn{\end{equation}}
\newcommand\beqnstar{\begin{equation*}}
\newcommand\eeqnstar{\end{equation*}}
\numberwithin{equation}{section}
\newcommand\mtitle[1]%
{\marginpar{\sffamily\raggedright\upshape\small #1}}

\providecommand\finalized{yes}
\ifthenelse{\equal{\finalized}{yes}}{}{\marginparwidth=2cm}
\ifthenelse{\equal{\finalized}{yes}}%
{\newcommand\mylabel[1]{\label{#1}}}%
{\newcommand\mylabel[1]{\label{#1}\marginpar{[{\ttfamily\upshape\tiny #1}]}}}
\ifthenelse{\equal{\finalized}{yes}}%
{\newcommand\checked[1]{}}%
{\newcommand\checked[1]{\marginpar{[{\ttfamily\upshape\tiny CHECKED: #1}]}}}
\ifthenelse{\equal{\finalized}{yes}}%
{\newcommand\spellchecked[1]{}}%
{\newcommand\spellchecked[1]{\marginpar{[{\ttfamily\upshape\tiny SPELLCHECKED: #1}]}}}
\providecommand\version{public}   %
\ifthenelse{\equal{\version}{public}}%
{\newcommand\mcomment[1]{}}%
{\newcommand\mcomment[1]{\marginpar{{\raggedright\sffamily\upshape\small
\begin{spacing}{0.75} #1\end{spacing}}}}}
\ifthenelse{\equal{\version}{public}}%
{\newcommand\fcomment[1]{}}%
{\newcommand\fcomment[1]{\footnote{#1}}}
\ifthenelse{\equal{\version}{public}}%
{\newcommand\comment[1]{}}%
{\newcommand\comment[1]{{\small #1}}}
\usepackage{graphicx}
\graphicspath{ {diagram/} }

\newcounter{diagram}
\numberwithin{diagram}{section}

\newenvironment{diagram}
{\stepcounter{diagram}\par\smallskip\noindent\begin{minipage}{\linewidth}\centering}
	{\par Diagram~\thediagram\end{minipage}\par\smallskip}

\choosefont{mathpazo}
\usepackage{youngtab}
\usepackage{amsmath,amsthm,xspace,xcolor}
\usepackage{cite,pgf,pgfarrows,pgfnodes,pgfautomata,pgfheaps}
\usepackage{tikz}
\usepackage{placeins}
\usepackage{hyperref}

\newcommand\algA{\lieg(A)}

\newcommand\lieg{\mathfrak{g}}

\newcommand\lieh{\mathfrak{h}}

\newcommand\pisys{$\pi$-system\xspace}
\newcommand\pisystem{\pisys}
\newcommand\pisystems{$\pi$-systems\xspace}
\newcommand{\qsig}[1][\Sigma]{q_{{}_{#1}}}

\newcommand{\Xsh}[1][X]{X_{\mathrm{short}}}
\newcommand{\Xlong}[1][X]{X_{\mathrm{long}}}
\newcommand{\bbar}[1][\beta]{\overline{#1}}

\newcommand{\roots}[1][]{\Delta_{#1}}
\newcommand{\reroots}[1][]{\Delta^{re}_{#1}}
\newcommand{\imroots}[1][]{\Delta^{im}_{#1}}

\newcommand{\rerootsh}{\reroots[\mathrm{short}]}
\newcommand{\rerootlong}{\reroots[\mathrm{long}]}
\newcommand{\Ext}{\mathrm{Ext}}

\newcommand{\form}[2]{\left( #1 \mid #2 \right)}
\newcommand{\aform}[2]{\langle\, #1, #2 \,\rangle}

\newcommand{\integers}{\mathbb{Z}}
\newcommand{\reals}{\mathbb{R}}
\newcommand{\complex}{\mathbb{C}}

\newcommand{\be}{\begin{enumerate}}
\newcommand{\ee}{\end{enumerate}}

\newcommand{\gcmset}{\mathbb{G}}

\newcommand{\al}{\alpha}

\begin{document}

\author{K. N. Raghavan}
\address{Krea University, Sri City, A.P. 517646}
\email{raghavan.komaranapuram@krea.edu.in}

\author{Krishanu Roy}
\address{SRM University AP, Andhra Pradesh, India, 522502}
\email{krishanu.r@srmap.edu.in}

\author{Sankaran Viswanath}
\address{The Institute of Mathematical Sciences, A CI of Homi Bhabha National Institute, Chennai 600113, India}
\email{svis@imsc.res.in}

\thanks{KNR and SV acknowledge support from DAE under a XII plan project. KR is partially supported by ISF grant no. 1221/17}

\title{On ${\boldsymbol\pi}$-systems of symmetrizable Kac-Moody algebras}

\thanks{}

\begin{abstract}
Given a symmetrizable Kac-Moody algebra $\lieg$, we study its \pisystems, which are subsets of real roots, the pairwise
differences of whose elements are not roots. Such systems arise as simple systems of regular subalgebras of $\lieg$, and were originally studied by Dynkin, Morita and Naito. We show that the binary relation introduced by Morita defines a partial order on the set of $\lieg$ of finite, untwisted affine or hyperbolic type. We also formulate general principles for constructing $\pi$-systems as well as for finding forbidden diagrams that cannot occur as Dynkin diagrams of \pisystems of a given $\lieg$. Among other applications, we use this to determine the set of maximal  hyperbolic Dynkin diagrams in ranks $3$-$10$ relative to the Morita partial order.
\end{abstract}
\maketitle

\section{Introduction}

Let $\lieg = \lieg(A)$ be the Kac-Moody algebra associated with a symmetrizable generalized Cartan matrix (GCM) $A$. Let $\Delta$ be its set of roots and $\Delta^{\mathrm{re}}$ be the set of real roots.
A \pisystem of $\lieg$ is a subset $\Sigma \subset \Delta^{\mathrm{re}}$ satisfying the property that $\alpha - \beta \not\in \roots$ for all $\alpha \neq \beta \in \Sigma$.

These were first studied by Dynkin for finite-dimensional $\lieg$ and later by Morita \cite{morita} and Naito \cite{naito} in general (see also \cite{feingold-nicolai}). Our previous work \cite{CRRRV} studied the Weyl group action on \pisystems and showed that the number of orbits is finite in various cases of interest.

In this paper, we study the binary relation  $\preceq$ on the set $\gcmset$ of symmetrizable generalized Cartan matrices (identifying two matrices which differ only by a simultaneous reordering of rows and columns). This was first introduced by Morita: given $A,B \in \gcmset$, we say $B \preceq A$ if there is a linearly independent \pisystem of type $B$ in the Kac-Moody algebra $\algA$ (see \S\ref{sec:prelims} for definitions). We establish that $\preceq$ defines a partial order on the set of GCMs of finite, untwisted affine or hyperbolic type (Proposition~\ref{prop:mainprop}), and conjecture that this holds on all of $\gcmset$.

We prove a very general theorem (Theorem~\ref{thm:pisys-mod-d}), from which we deduce necessary conditions for a pair $(A,B)$ to satisfy $B \preceq A$ under various hypotheses on $A$. This in turn allows us to find {\em forbidden diagrams} of \pisystems, i.e., those which cannot occur as subdiagrams of Dynkin diagrams of \pisystems of $A$. We give many examples and applications of this in \S\ref{Necessary}. 

We also prove some general principles for explicit construction of \pisystems; these may be viewed as generalizations of principles obtained in \cite{svis-e10} for simply-laced diagrams. As an application of these constructions and the forbidden diagram analysis, we show how to find the list of all maximal Dynkin diagrams in the set of all hyperbolic Dynkin diagrams. An approach to the poset of hyperbolic type GCMs under $\preceq$ is contained in \cite{felikson-tumarkin,tumarkin}, but proceeds via the associated Weyl groups.

\section{Preliminaries}\label{sec:prelims}

We recall the relevant notation from \S 2 of \cite{CRRRV}.
  \subsection{}An integer matrix $A=(a_{ij})$ of size $n\times n$, where $n$ is a positive integer, is called a {\em generalized Cartan matrix\/}, {\em GCM\/} for short, if the following conditions are satisfied:
\begin{enumerate}
	\item $a_{ii}=2$ for all $1\leq i\leq n$
	\item $a_{ij}\leq 0$ whenever $1\leq i, j\leq n, i \neq j$
	\item $a_{ij}=0$ if %
	$a_{ji}=0$ for $1\leq i, j\leq n$
\end{enumerate}

Given a GCM $A$ of size $n$,  we let $\algA$ denote the {\em Kac-Moody Lie algebra\/} associated to $A$ \cite[\S1.3]{kac}, with Cartan subalgebra $\lieh(A)$ and Chevalley generators $e_i, f_i$ for $1 \leq i \leq n$.  Let $\alpha_i(A), 1 \leq i \leq n$ denote the simple roots of $\algA$ and let $Q(A)$ be its root lattice. We use terminology and notation as in the early chapters of~\cite{kac} without any further comment.

\bdefn\mylabel{d:pisys}  Let $A$ be a GCM.
A {\em \pisystem} in $A$ is a finite collection of distinct real roots $\{\beta_i\}_{i=1}^m$ of $\lieg(A)$ such that $\beta_i-\beta_j$ is not a root for any $1\leq i\neq j\leq m$.
\edefn

Let $A$ be a GCM, and $\Sigma = \{\beta_i\}_{i=1}^m$ be a \pisystem in $A$.
Then the matrix \[ M(\Sigma) := \left[ \aform{\beta_i^\vee}{\beta_j} \right]_{i,j=1}^m \] is a GCM. Here, we let $\alpha^\vee$ denote the coroot corresponding to $\alpha$ when $\alpha$ is a real root of $\algA$. 
\bdefn
We call $B:=M(\Sigma)$ the {\em type} of $\Sigma$, and refer to $\Sigma$ as a {\em \pisystem of type $B$ in $A$\/}.
\edefn

\subsection{Symmetrizable GCMs and \pisystems}\label{sec:symm-gcm}
An $n\times n$ GCM $A$ is {\em symmetrizable\/} if there exists a diagonal $n\times n$ matrix $D$ with positive rational diagonal entries such that $DA$ is symmetric. %
Let $\Sigma =\{\beta_i: 1 \leq i \leq m\}$ be a \pisystem of type $B$ in $A$. We note that if $A$ is a symmetrizable GCM, then so is $B$. Fix a choice of diagonal matrix $D$ which symmetrizes $A$, and let  $\form{\cdot}{ \cdot}$ denote the corresponding symmetric bilinear form  on $\lieh(A)$ \cite[\S 2.1]{kac}, such that:
\begin{equation} \label{eq:form-def} \form{\alpha_i^{\vee}(A)}{\alpha_j^{\vee}(A)} = \,a_{ij}/D_{jj} \end{equation}

 Since the bilinear form $\form{\cdot}{ \cdot}$ is non-degenerate on $\lieh(A)$, we have an isomorphism $\nu : \lieh(A) \rightarrow \lieh^*(A)$ defined by
 \begin{equation}
 	\nu(h)(h_1)=\form{h}{h_1},  \end{equation} for $ h,h_1\in \lieh(A)$
 and the induced bilinear form $\form{\cdot}{ \cdot}$ on $\lieh^*(A)$ such that 
 \begin{equation} \label{eq:form-def} \form{\alpha_i(A)}{\alpha_j(A)} =D_{ii} \,a_{ij} \end{equation}

Since the $\beta_i$ are real roots of $\algA$, we know by \cite[Chapter 5]{kac} that:
\[b_{ij} = \langle \beta_i^\vee,\beta_j\rangle = \frac{2\form{\beta_i}{\beta_j}}{ \form{\beta_i}{\beta_i} }\]
Thus, $D' = \mathrm{diag}( \form{\beta_i}{\beta_i}/2)$ is a diagonal matrix with positive rational entries that symmetrizes $B$.  This choice of symmetrization defines a symmetric bilinear form on $Q(B)\otimes_\integers \complex$. As in equation~\eqref{eq:form-def} above, this is given by $\form{\alpha_i(B)}{\alpha_j(B)} = D'_{ii}\, b_{ij} = \form{\beta_i}{\beta_j}$. In other words, given the compatible choices of symmetrizations $(D,D')$ as above, the $\complex$-linear map
\begin{equation} \label{eq:def-q}
	Q(B) \otimes_\integers \complex \to Q(A) \otimes_\integers \complex, \;\;\;\;\alpha_i(B) \mapsto \beta_i \text{ for } 1 \leq i \leq m
\end{equation}
is form preserving.

\section{The relation $\preceq$}\label{sec:partial-order}
Consider the set $\gcmset$ of symmetrizable GCMs (of all sizes). We identify two such GCMs if they are equal upto a simultaneous permutation of rows and columns.

For $A, B \in \gcmset$, we define the relation $B \preceq A$ if there is a {\em linearly independent} \pisystem of type $B$ in $A$. Clearly this relation is reflexive. By corollary 2.14 in \cite{CRRRV} this relation is transitive. We however do not know if this relation is anti-symmetric.

\subsection{} In this section, we derive some key properties of $\preceq$. Later, in \S\ref{sec:fah}, we use these to establish its anti-symmetry in many important cases, in particular, when one of the matrices is of finite, untwisted affine or hyperbolic type.

\begin{lemma}\label{lem:det-div}
  Let $A$ be an $n \times n$ symmetrizable GCM. Let $\{\alpha_i\}_{i=1}^n$ be the simple roots of $\lieg(A)$ and $\alpha^\vee_i$ be the corresponding coroots. Let  $\{\beta_i\}_{i=1}^n \subset Q(A)$ and $\{\gamma^\vee_i\}_{i=1}^n \subset Q^\vee(A)$ be any $n$-element subsets of the root and coroot lattices respectively. Consider the integer matrix:
  $B = \left[ \aform{\gamma^\vee_i}{\beta_j} \right]_{ij}$.
  Then:
  \be
\item $\det A$ divides $\det B$ (in $\integers$).
\item Further if $A, B$ are invertible with $|\det A| = |\det B|$, then $\{\beta_i\}_{i=1}^n$ and $\{\gamma_i^\vee\}_{i=1}^n$ form $\integers$-bases of $Q(A)$ and $Q^\vee(A)$ respectively.
\ee
\end{lemma}
\bmyproof
We write:
\begin{align}
  \gamma^\vee_i &= \sum_{k=1}^n u_{ik} \, \alpha^\vee_k \label{eq:coroot}\\
  \beta_j &= \sum_{\ell=1}^n v_{j\ell} \, \alpha_\ell \label{eq:root}
\end{align}
where $u_{ik}, v_{j\ell}$ are integers. Using the equations above, we compute:
\[ B = U A V^T \]
where $U = \left[ u_{ij} \right]$ and $V = \left[ v_{ij} \right]$ are integer matrices. Taking determinants, we obtain $\det B = \det U \det V \det A$, proving the first assertion. For the second assertion, the given condition implies $|\det U| = |\det V| =1$, i.e., $U$ and $V$ are in $\mathrm{GL}_n(\integers)$. This is clearly equivalent to what needs to be shown.
\emyproof

\begin{proposition}\label{prop:lemcor}
  Let $A, B \in \gcmset$ have sizes $n, m$ respectively and let $B \preceq A$.  Then:
  \begin{enumerate}
  \item $m \leq n$.
  \item If $m=n$, then $\det B =k \det A$ for some integer $k \geq 1$.
  \item If $m=n$, then the symmetrizations of $A$ and $B$ have the same signature.
  \end{enumerate}
\end{proposition}

\begin{proof}
Let $S = \{\alpha_i\}_{i=1}^n$ denote the set of simple roots of $\lieg(A)$. 
The hypothesis implies that there is a linearly independent \pisystem $\Sigma$ of type $B$ in $A$. The linear independence of $\Sigma$ clearly implies that $m \leq n$. Now if $m=n$, let $\Sigma = \{\beta_i\}_{i=1}^n$. In Lemma~\ref{lem:det-div}, we set $\gamma^\vee_i:=\beta^\vee_i = 2 \nu^{-1}(\beta_i)/ |\beta_i|^2$ (i.e, the corresponding coroots). The linear independence of $\Sigma$ ensures that the matrices $U, V$ of Lemma~\ref{lem:det-div} are invertible over $\reals$. Further,  from equations \eqref{eq:coroot} and \eqref{eq:root}, we conclude that $v_{ik} = u_{ik}|\beta_i|^2/|\alpha_k|^2$\cite[Eq. (5.1.1)]{kac}. Thus the matrices $U, V$ are related by $V = D_1 U D_2$ where $D_1, D_2$ are the diagonal matrices with entries $|\beta_i|^2$ and $|\alpha_k|^2$ respectively. Since these entries are positive rationals, we conclude that $\det V = p \det U$ for some positive rational $p$, and thus $\det B = k \det A$ where $k = p(\det U)^2 >0$. Since $k$ is known to be an integer from Lemma~\ref{lem:det-div}, this proves the second assertion. For the third assertion, we observe that the signature of the form $\form{\cdot}{\cdot}$ on the $\reals \Sigma$ (resp. $\reals S$) is exactly the signature of the symmetrization of $B$ (resp. $A$). But  since $m=n$, we have $\reals\Sigma = \reals S$.
\end{proof}

\begin{corollary} \label{cor:cor-anti}
Let  $A, B \in \gcmset$ be indecomposable GCMs such that $A \preceq B$ and $B \preceq A$. Then they have the same sizes, same types (Finite/Affine/Indefinite) and $\det A = \det B$. 
\end{corollary}

\begin{remark}
Given $A, B \in \gcmset$ with $B \preceq A$, it is natural to ask if $B^T \preceq A^T$ ? However, this turns out to be false even in finite type, as can be seen from Dynkin's tables\cite[Tables 9, 11]{dynkin}. We quote one example: $C_2 \oplus C_2 \preceq C_4$, but $B_2 \oplus B_2 \not\preceq B_4$. To see this, recall that the root system $C_4$ may be realized as the set of vectors of norm 1 or 2  in the rank 4 lattice $\oplus_{i=1}^4 \integers \epsilon_i$ where $\epsilon_i$ are the standard orthonormal basis of $\reals^4$\cite[\S 12.1]{hum}. We now observe that $\{\epsilon_1 - \epsilon_4, \epsilon_4\} \cup \{\epsilon_2-\epsilon_3, \epsilon_3\}$ defines a \pisystem of type $C_2 \oplus C_2$ in $C_4$.  But in case of the dual root systems, the short simple roots of $B_2 \oplus B_2$ form a subdiagram of type $A_1 \times A_1$, and Lemma~\ref{lem:da2-hyp} below implies that $B_2 \oplus B_2 \not\preceq B_4$.
\end{remark}

\noindent
{\bf{Examples:}}  Diagram~\ref{fig: diags}, shows the {\em overextended} Dynkin diagrams $A_n^{++}$ and $D_n^{++}$ \cite[\S 5.1]{CRRRV}, and the diagram $E_n$. For $n \geq 8$, we note that $A_n^{++}$, $D_n^{++}$ and $E_{n+2}$ all have $n+2$ vertices and are of Lorentzian signature. We further know (for instance, from \cite{svis-e10}) that $A_8^{++} \preceq E_{10}$ and $D_8^{++} \preceq E_{10}$. 

             The determinants of these GCMs may be readily computed:
             \[ \det(A_n^{++}) = -(n+1), \;\; \det(D_n^{++}) = -4, \;\; \det(E_{n+2}) = 7-n \]

             It follows from Proposition~\ref{prop:lemcor}(2) that $A_n^{++} \not\preceq E_{n+2}$ for $n \geq 16$ and $D_n^{++} \not\preceq E_{n+2}$ for $n \geq 12$. 
 
\bigskip
\noindent
\begin{diagram}\label{fig: diags}
	\begin{tikzpicture}[scale=.29]
		
		\begin{scope}[shift={(-5,0)}]
			\draw (12.5,-3.5) node{$D_n^{++}$};
			\draw (5,0) node[below] {} circle (.3cm);
			\draw (8,0) node[below] {} circle (.3cm);
			\draw (11,0) node[below] {} circle (.3cm);
			\draw (18.5,0) node[below] {} circle (.3cm);
			\draw (21.5,0) node[below] {} circle (.3cm);
			\draw (11,3) node[below] {} circle (.3cm);
			\draw (18.5,3) node[below] {} circle (.3cm);
			\draw (14.2,0) node[below] {} circle (.08cm);
			\draw (14.8,0) node[below] {} circle (.08cm);
			\draw (15.4,0) node[below] {} circle (.08cm);
			
			\draw (5.3,0) -- +(2.4,0);
			\draw (8.3,0) -- +(2.4,0);
			\draw (11.3,0) -- +(2.4,0);
			\draw (11,0.3) -- +(0,2.4);
			\draw (15.8,0) -- +(2.4,0);
			\draw (18.8,0) -- +(2.4,0);
			\draw (18.5,0.3) -- +(0,2.4);

		\end{scope}
		\begin{scope}[shift = {(16,0)}]
			\draw (12.5,-3.5) node{$A_n^{++}$};
			\draw (5,0) node[below] {} circle (.3cm);
			\draw (8,0) node[below] {} circle (.3cm);
			\draw (11,0) node[below] {} circle (.3cm);
			\draw (14.9,0) node[below] {} circle (.08cm);
			\draw (15.5,0) node[below] {} circle (.08cm);
			\draw (16.1,0) node[below] {} circle (.08cm);
			\draw (20,0) node[below] {} circle (.3cm);
			\draw (11,3) node[below] {} circle (.3cm);
			
			\draw (5.3,0) -- +(2.4,0);
			\draw (8.3,0) -- +(2.4,0);
			\draw (11.3,0) -- +(2.4,0);
			\draw (11,0.3) -- +(0,2.4);
			\draw (17.3,0) -- +(2.4,0);

			\draw    (5,0) to[out=-30,in=-150] (20,0);

		\end{scope}
		
		\begin{scope}[shift = {(35,0)}]
			\draw (12.5,-3.5) node{$E_n$};
			\draw (5,0) node[below] {} circle (.3cm);
			\draw (8,0) node[below] {} circle (.3cm);
			\draw (11,0) node[below] {} circle (.3cm);
			\draw (14.9,0) node[below] {} circle (.08cm);
			\draw (15.5,0) node[below] {} circle (.08cm);
			\draw (16.1,0) node[below] {} circle (.08cm);
			\draw (20,0) node[below] {} circle (.3cm);
			\draw (11,3) node[below] {} circle (.3cm);
			
			\draw (5.3,0) -- +(2.4,0);
			\draw (8.3,0) -- +(2.4,0);
			\draw (11.3,0) -- +(2.4,0);
			\draw (11,0.3) -- +(0,2.4);
			\draw (17.3,0) -- +(2.4,0);

		\end{scope}
		\end{tikzpicture}
\end{diagram}

\section{Finite, Affine and Hyperbolic}\label{sec:fah}
\subsection{}
When one of the GCMs is of finite, untwisted affine or hyperbolic type, we can strengthen the conclusion of Corollary~\ref{cor:cor-anti} as follows.
\begin{proposition}\label{prop:mainprop}
  Let $A$ be a GCM of finite, untwisted affine or hyperbolic type and let $B \in \gcmset$ such that $A \preceq B \preceq A$. Then $B$ coincides with $A$ upto a simultaneous permutation of rows and columns.
\end{proposition}
 In other words, $\preceq$ is a partial order on the set of such GCMs (identifying GCMs that differ by a simultaneous permutation of rows and columns). We conjecture below that Proposition~\ref{prop:mainprop} holds for all symmetrizable GCMs. In particular when $A, B$ are of twisted affine type, a possible approach is via an appeal to the Tables in \cite{FRT} or \cite{roy-venkatesh}.   
\begin{conjecture}
  The relation $\preceq$ is a partial order on all of $\gcmset$.
\end{conjecture}

\noindent
    {\em Proof of Proposition~\ref{prop:mainprop}:} Observe that by Corollary~\ref{cor:cor-anti}, the types (Finite/Affine/Indefinite) of $A$ and $B$ are the same and $\det A = \det B$. If $A$ and $B$ are both of finite type, then by \cite[Proposition 2.6]{CRRRV}, there are injective Lie algebra homomorphisms $\algA \to \lieg(B) \to \algA$. The composition is injective, and by finite-dimensionality, also surjective. Thus, $\lieg(A) \cong \lieg(B)$ and it follows that $A$ and $B$ are related by a simultaneous permutation of rows and columns (for instance by \cite{PetersonKac}); an alternate proof is an exhaustive check using the Tables of \cite{dynkin}.

    Next, suppose that $A$ and $B$ are of affine type. Since $A$ is untwisted affine, let $\overline{A}$ denote the underlying finite type GCM. Theorem 4.2 of \cite{naito} implies that $B \preceq A$ iff $B$ is also untwisted affine with $\overline{B} \preceq \overline{A}$, where $\overline{B}$ is the underlying finite type GCM of $B$. Thus $A \preceq B \preceq A$ implies $\overline{A} \preceq \overline{B} \preceq \overline{A}$ and the conclusion follows from the finite case above.

    If $A$ and $B$ are both of hyperbolic type, then the assertion follows from  proposition~\ref{prop:po} below.  \qed

\begin{proposition}\label{prop:po}
  Let $A, B$ be $n \times n$ symmetrizable GCMs of hyperbolic type, with $\det A = \det B$ and $B \preceq A$. Suppose $\Sigma = \{\beta_i\}_{i=1}^n$ is a \pisystem of type $B$ in $A$. Then
  \begin{enumerate}
  \item $\Sigma$ is $W(A)$-conjugate to $\Pi(A)$ or $-\Pi(A)$, where $\Pi(A)$ is the set of simple roots of $\lieg(A)$.
  \item In particular, $A$ and $B$ are equal up to a simultaneous permutation of rows and columns.
    \end{enumerate}
\end{proposition}
\bmyproof
Consider the map $\qsig: Q(B) \to Q(A)$ of equation~\eqref{eq:def-q}, defined by $\alpha_i(B) \mapsto \beta_i$ for all $i$, where $\Pi(B) = \{\alpha_i(B): 1 \leq i \leq n\}$ is the set of simple roots of $\lieg(B)$. We assume for convenience that the symmetric bilinear forms on $Q(A)$ and $Q(B)$ are chosen compatibly as in \S\ref{sec:symm-gcm}, so that $\qsig$ is form preserving (the arguments below will still work for any choices of standard invariant forms, since they only differ by scaling by positive rationals).

Using the given hypothesis and the fact that hyperbolic GCMs are necessarily invertible, we obtain from the second part of lemma~\ref{lem:det-div} that: $(i)$ $\Sigma$ is a $\integers$-basis of $Q(A)$ and $(ii)$  $\Sigma^\vee=\{\beta^\vee_i\}_{i=1}^n$ is a $\integers$-basis of $Q^\vee(A)$.

We observe that $\qsig$ is a form preserving lattice isomorphism of $Q(B)$ onto $Q(A)$.  We now claim that $\qsig(\roots(B)) = \roots(A)$. By Corollary 2.12 in \cite{CRRRV}, we know that $\qsig(\roots(B)) \subset \roots(A)$. We only need to prove the reverse inclusion.
Towards this end, we recall the following description of the set of roots of a symmetrizable Kac-Moody algebra $\lieg(C)$ of Finite, Affine or Hyperbolic type \cite[Prop 5.10]{kac}:
\begin{align}
  \reroots(C) &= \{ \alpha=\sum_j k_j \,\alpha_j(C)  \,\in Q(C): |\alpha|^2 >0 \text{ and } k_j\,|\alpha_j(C)|^2/|\alpha|^2 \in \integers \text{ for all } j\} \label{eq:reroot-hyp}\\
  \imroots(C) &= \{\alpha \in Q(C) \backslash \{0\}: |\alpha|^2 \leq 0\} \label{eq:imroot-hyp}
\end{align} forms
where $\alpha_j(C)$ are the simple roots, $Q(C)$ is the root lattice, and we fix any standard invariant form on $\lieg(C)$.
We apply this when $C=A, B$ below.

\smallskip
Since $|\qsig(\alpha)|^2 = |\alpha|^2$ for all $\alpha \in Q(B)$, it is clear from equation~\eqref{eq:imroot-hyp} that $\qsig(\imroots(B)) = \imroots(A)$. Now let $\beta \in \reroots(A)$ and define $\alpha = \qsig^{-1}(\beta)$. We need to prove that $\alpha \in \reroots(B)$.
Let $\beta = \sum_j k_j \beta_j$ for some integers $k_j$; thus $\alpha = \sum_j k_j \, \alpha_j(B)$. Since $\beta$ is a real root, $|\alpha|^2 = |\beta|^2 >0$.
Define \[c_j = k_j \, |\alpha_j(B)|^2 / |\alpha|^2 = k_j \,|\beta_j|^2 / |\beta|^2 \]
Equation~\eqref{eq:reroot-hyp} implies that $\alpha$ is a real root of $\lieg(B)$ if and only if $c_j \in \integers$ for all $j$. Consider $\beta^\vee \in Q^\vee(A)$; by $(ii)$ above, we know that $\Sigma^\vee$ forms a $\integers$-basis of the coroot lattice $Q^\vee(A)$. Now $\gamma^\vee =2 \nu^{-1}(\gamma)/ |\gamma|^2$ for any real root $\gamma$ of $\lieg(A)$ \cite[Prop. 5.1]{kac}, where $\nu$ is the linear isomorphism from the Cartan subalgebra of $\lieg(A)$ to its dual induced by the form.  A simple computation now shows :
\[ \beta^\vee = \sum_j c_j \,\beta^\vee_j \]
This proves the integrality of the $c_j$, and hence our claim.
Thus, $\qsig(\roots(B)) = \roots(A)$. Since $\qsig(\Pi(B)) = \Sigma$, this means that $\Sigma$ is a {\em root basis} of $\roots(A)$ \cite[\S 5.9]{kac}, i.e., $\Sigma$ is a $\integers$-basis of $Q(A)$ such that every element of $\roots(A)$ can be expressed as an integral linear combination of $\Sigma$ with all coefficients of the same sign. By \cite[Proposition 5.9]{kac}, we conclude that $\Sigma$ is $W(A)$-conjugate to $\pm \Pi(A)$.
Finally, since $\Pi(A)$ is a \pisystem of type $A$ in $A$, we conclude that $A=B$, up to a simultaneous permutation of rows and columns.
\emyproof

Propositions~\ref{prop:po} and \ref{prop:lemcor} imply the following useful lemma:
\begin{lemma}\label{lem:det-multiple} {\rm (Determinant criterion)}
Let $A, B$ be symmetrizable hyperbolic GCMs of the same size. If $B \preceq A$ and $B \neq A$ (up to simultaneous reordering of rows and columns), then $\det B = k \det A$ for some $k \geq 2$.
\end{lemma}

\section{Necessary conditions for $B \preceq A$}\label{Necessary}

\subsection{}
The following is an immediate corollary of the discussion of \S\ref{sec:symm-gcm}, together with the fact that a real root is Weyl conjugate to some simple root, and therefore has the same length.
\begin{lemma}\label{lem:rlr} {\rm (Root length criterion)} Let $A,B$ be indecomposable symmetrizable GCMs such that $B \preceq A$. For each pair of simple roots of $B$, the ratio of their lengths  equals that of some pair of simple roots of $A$ (with respect to any choices of standard invariant forms on $\lieg(A)$ and $\lieg(B)$).
\end{lemma}

For instance, this implies that there doesn't exist a \pisystem of type $G_2$ in any other finite type GCM.

\subsection{}
Let $X$ be the Dynkin diagram of a symmetrizable Kac-Moody algebra and let $W$ denote its Weyl group. We define $\Xsh$ to be the subdiagram formed by the simple roots of shortest length, i.e,
\[\Xsh = \{p \in X: |\alpha_p| = \min_{i \in X} |\alpha_i|\}\]
Similarly $\Xlong$ is the subdiagram formed by the simple roots of longest length. We also let
\[\rerootsh(X) =\{\alpha \in \reroots(X): |\alpha| = \min_{i \in X} |\alpha_i|\} = W\cdot\Xsh\]
and $\rerootlong(X) = W\cdot \Xlong$.
We say $X$ is {\em doubly-laced} (resp. {\em triply-laced}) if (i) every edge of $X$ is either a single edge or a double (resp. triple) edge with an arrow pointing in one direction, and (ii) $X$ contains at least one double (resp. triple) edge.  For instance, the finite type diagrams $B_n, C_n, F_4$ are doubly-laced while $G_2$ is triply-laced. The next lemma is a direct consequence of these definitions.

\begin{lemma}\label{lem:shortlong-d-div}
  Let $X$ be a doubly- or triply-laced Dynkin diagram of a symmetrizable Kac-Moody algebra (we set $d=2$ in the former case, $d=3$ in the latter). Then:
  \be
  \item $d \mid \aform{\alpha_i^\vee}{\alpha_j}$ for all $i \in \Xsh, \, j \in X\backslash \Xsh$.
\smallskip\item $d \mid \aform{\alpha_j^\vee}{\alpha_i}$ for all $i \in \Xlong,\, j \in X\backslash \Xlong$.\qed
  \ee
\end{lemma}

Now consider \pisystems $\Sigma$ in $X$ such that $\Sigma \subset \rerootsh$ or $\Sigma \subset \rerootlong$. We seek to understand the possible types of such $\Sigma$. The proposition of the next subsection is the important result that will enable us to answer this question. This proposition is vastly more general and can be applied to a wide variety of settings.

\subsection{} Let $X$ be the Dynkin diagram of a symmetrizable Kac-Moody algebra and let $Y$ be a subdiagram of $X$. We let $\roots(Y)$ denote the set of roots of $Y$, and identify it with $Q(Y) \cap \roots(X)$ where $Q(Y) = \displaystyle\bigoplus_{i \in Y} \integers \alpha_i$. Let $W$ denote the Weyl group of $X$. The following proposition concerns multisubsets $\Sigma$ of the set $W\cdot \reroots(Y) = \displaystyle\bigcup_{p \in Y} W\alpha_p$. We recall also the notation $M(\Sigma)$ from \S\ref{d:pisys}.

\begin{theorem}\label{thm:pisys-mod-d}
    Let $X$ be the Dynkin diagram of a symmetrizable Kac-Moody algebra, $Y$ a subdiagram of $X$ and $d \geq 2$ an integer.
Suppose that either:
\begin{align}
  d \mid \aform{\alpha_j^\vee}{\alpha_i} & \text{ for all } i \in Y,\, j \in X\backslash Y, \text{ or } \label{eq:ddiv1}\\
  d \mid \aform{\alpha_i^\vee}{\alpha_j} & \text{ for all } i \in Y, \, j \in X\backslash Y.\label{eq:ddiv2} 
\end{align}
Let $\Sigma = \{\beta_i: 1 \leq i \leq m\}$ be a multiset with $\beta_i \in W\cdot \reroots(Y)$. Then, there exists a multiset $\bbar[\Sigma] =  \{\bbar_i: 1 \leq i \leq m\}$ with $\bbar_i \in \reroots(Y)$ such that
\[M(\Sigma) \equiv M(\bbar[\Sigma]) \pmod{d}\] 
\end{theorem}
\bmyproof
Let $s_i$ denote the simple reflection corresponding to the vertex $i \in X$ and let $W(Y)$ be the (standard parabolic) subgroup of $W$ generated by the $\{s_i: i \in Y\}$. The given hypothesis implies by \cite[Prop 3.13]{kac} that for each $i \in Y, \, j \in X \backslash Y$, $(s_is_j)^{m_{ij}} =1$ where $m_{ij} = 2,4,6$ or $\infty$. Since these are even (or $\infty$), it follows that the map  $W \to W(Y)$ defined on the generators by:
\[s_i \mapsto \begin{cases} s_i & i \in Y \\ 1 & i \in X\backslash Y \end{cases}\]
extends to a group homomorphism. We denote it $w \mapsto \bbar[w]$.
Let $Q(X), \;Q^\vee(X)$ denote the root and coroot lattices of $X$. We define sublattices $R, R^\vee$ as follows. If \eqref{eq:ddiv1} holds, then $R := d\,Q(X)$, and
\[R^\vee :=   d\,Q^\vee(Y) \oplus Q^\vee(X\backslash Y) = \bigoplus_{i \in Y} \; \integers \,(d\alpha^\vee_i) \; \oplus \;\bigoplus_{j \not\in Y}\; \integers \alpha^\vee_j. \]
If \eqref{eq:ddiv2} holds, then
\[R :=   d\,Q(Y) \oplus Q(X\backslash Y) \;\;\; \text{ and } \;\;\; R^\vee = d\,Q^\vee(X).\]

\smallskip
The given hypotheses readily imply that $R$ and $R^\vee$ are $W$-invariant. 
We now make the following important observation:
\begin{equation}\label{eq:imp-obs}
\text{Given } (w,\alpha)  \in W \times \reroots(Y), \text{ we have } w\alpha \in \bbar[w]\alpha + R \text{ and } w(\alpha^\vee) \in \bbar[w](\alpha^\vee) + R^\vee
\end{equation}
It is enough to prove this on the generators $w=s_k$ of $W$. This is obvious when $k \in Y$ and follows from equations~\eqref{eq:ddiv1}, \eqref{eq:ddiv2} when $k \in X\backslash Y$.
Now, given $\beta \in W\cdot \reroots(Y)$, say $\beta = \sigma \alpha$ for some $(\sigma, \alpha) \in W \times \reroots(Y)$, we define $\bbar := \bbar[\sigma] \alpha$. This is a real root of $Y$, and in view of \eqref{eq:imp-obs} above, the association $\beta \mapsto \bbar$ is well-defined modulo $R$. Further, if $\gamma = \tau \alpha'$ is another root in the $W$-orbit of $\reroots(Y)$, then
\begin{equation} \label{eq:form-mod-d}
  \aform{\,\bbar^\vee}{\bbar[\gamma]\,} = \aform{\,\bbar[\sigma] (\alpha^\vee)}{\,\bbar[\tau] \alpha'\,} \equiv \aform{\,\sigma (\alpha^\vee)}{\,\tau \alpha'\,} \pmod{d}
\end{equation}
The congruence modulo $d$ in this equation is an easy consequence of equation~\eqref{eq:imp-obs}, together with the observations that
\begin{align*}
  \aform{Q^\vee(X)\,}{\,R} \equiv \aform{R^\vee\,}{\,Q(Y)} \equiv 0 \pmod{d} && \text{if equation~\eqref{eq:ddiv1} holds.}\\
  \aform{R^\vee\,}{\,Q(X)} \equiv \aform{Q^\vee(Y)\,}{\,R} \equiv 0 \pmod{d} && \text{if equation~\eqref{eq:ddiv2} holds.}
\end{align*}
Finally, if $\Sigma = \{\beta_i: 1 \leq i \leq m\}$ is a multi-subset of $W\cdot \reroots(Y)$, define $\bbar[\Sigma] =  \{\bbar_i: 1 \leq i \leq m\}$.
Equation~\eqref{eq:form-mod-d} now implies $M(\Sigma) \equiv M(\bbar[\Sigma]) \pmod{d}$ as required.
\emyproof

We obtain several useful corollaries.

\begin{corollary}\label{cor:da1}
Let $X$ be a doubly-laced Dynkin diagram of a symmetrizable Kac-Moody algebra. Suppose that $\Xsh$ (respectively $\Xlong$) is of type $A_1$, i.e., is a single vertex, then there is no \pisystem of type $A_2$ in $X$ contained wholly in $\rerootsh(X)$ (respectively $\rerootlong(X)$).
\end{corollary}

\begin{corollary}\label{cor:da2}
  Let $X$ be a doubly-laced Dynkin diagram of a symmetrizable Kac-Moody algebra. Suppose that $\Xsh$ (respectively $\Xlong$) is of type $A_2$, then there is no \pisystem
  of type $A_2 \times A_1$ in $X$ contained wholly in $\rerootsh(X)$ (respectively $\rerootlong(X)$).
\end{corollary}

\begin{corollary}\label{cor:ta1}
  Let $X$ be a triply-laced Dynkin diagram of a symmetrizable Kac-Moody algebra. Suppose that $\Xsh$ (respectively $\Xlong$) is of type $A_1$, then there is no \pisystem of type
  $A_1 \times A_1$ in $X$ contained wholly in $\rerootsh(X)$ (respectively $\rerootlong(X)$).
\end{corollary}

We indicate how to prove Corollary~\ref{cor:da2}, the others being similar. Lemma~\ref{lem:shortlong-d-div} allows us to apply Theorem~\ref{thm:pisys-mod-d} with $Y=\Xsh$ (or $\Xlong$) and $d=2$. The set of shortest (or longest) real roots of $X$ is nothing but $W \cdot \reroots(Y)$. Given any \pisystem (in fact any multiset of real roots) $\Sigma$ of $X$ contained wholly in the Weyl group orbit of $\reroots(Y)$, we obtain the multisubset $\bbar[\Sigma]$ of $\reroots(Y)$ such that $M(\Sigma)$ coincides with $M(\bbar[\Sigma])$ modulo $d=2$. For $Y$ of type $A_2$, it only remains to verify that no such multisubset exists if we take $M(\Sigma)$ to be the GCM of type $A_2 \times A_1$, i.e., the matrix
\[M =\begin{pmatrix} 2 & -1 & 0 \\ -1 &2 & 0 \\ 0 & 0 &2 \end{pmatrix}\]
So let $\bbar[\Sigma]=\{\bbar_1, \bbar_2, \bbar_3\}$ be such that $M(\bbar[\Sigma])$ is congruent to $M$ mod 2. We observe that the root system of type $A_2$ has the property that given two (real) roots $\alpha, \beta$, we have $\aform{\alpha^\vee}{\beta}$ is even if and only if $\beta = \pm \alpha$. Since the third row and column of $M$ is zero mod 2, we conclude that $\bbar_1$ and $\bbar_2$ must both be of the form $\pm \bbar_3$. But this would imply $\aform{\beta^\vee_1}{\beta_2}$ is also even, which is a contradiction. \qed

In many situations, corollaries~\ref{cor:da1}--\ref{cor:ta1} give us {\em forbidden diagrams} which cannot occur as subdiagrams of Dynkin diagrams of  \pisystems. For example, we have:

  \begin{proposition}\label{prop:forb}
    Let $X$ be a doubly-laced Dynkin diagram of a symmetrizable Kac-Moody algebra with exactly two real root lengths. Let $Y$ be the Dynkin diagram of a linearly independent \pisystem of $X$.
    \begin{enumerate}
    \item Suppose $\Xsh$ (respectively $\Xlong$) is of type $A_1$, then the finite diagram $C_3$ (respectively $B_3$) cannot occur as a subdiagram of $Y$.
    \item Suppose $\Xsh$ is of type $A_2$, then the finite diagram $C_5$ and the hyperbolic diagram $HD_3^{(2)}$ ($=\Gamma_{163}$ in Table I) cannot occur as subdiagrams of $Y$.
    \item Likewise, if $\Xlong$ is of type $A_2$, then the finite diagram $B_5$ and the hyperbolic diagram $HC_2^{(1)}$ ($=\Gamma_{162}$ in Table I) cannot occur as subdiagrams of $Y$.
    \end{enumerate}
  \end{proposition}

  The proof is a simple application of Corollaries~\ref{cor:da1} and \ref{cor:da2}. We demonstrate one of the cases and leave the rest to the reader. The diagram $C_3$ has two short simple roots which form an $A_2$ diagram. If $Y$ contained $C_3$ as a subdiagram (or even a \pisystem of $C_3$ type), it would imply that $X$ contains a \pisystem of $A_2$ type consisting only of short roots (since $X$ only has two root lengths). This contradicts corollary~\ref{cor:da1}.

  As is clear from the discussion above, Proposition~\ref{prop:forb} is only a small sample of the possibilities - one can write down many other forbidden diagrams in the doubly- and triply-laced cases. One can also replace the condition $\Xsh=A_1$ or $A_2$ in corollaries ~\ref{cor:da1} and \ref{cor:da2} by $\Xsh=A_n$ for other $n \geq 3$. The resulting analysis is more involved, but forbidden diagrams can be extracted in these cases too; for $n=4$ the configurations that one needs to consider are related to Petersen's graph on 10 vertices - this will be considered in greater detail in future work.

\medskip

\noindent
\begin{tikzpicture}[scale=.19]
	
	\begin{scope}[shift={(-5,0)}]
		\draw (11,0) node[below] {} circle (.3cm);
		\draw (14,0) node[below] {} circle (.3cm);
		\draw (17,0) node[below] {} circle (.3cm);
		\draw (20,0) node[below] {} circle (.3cm);
		\draw (11.3,0) -- +(2.4,0);
		\draw (14.5,0.15) -- +(2.25,0);
		\draw (14.5,-0.15) -- +(2.25,0);
		\draw (17.3,0.15) -- +(2.25,0);
		\draw (17.3,-0.15) -- +(2.25,0);
		
		\draw[thick] (19.7,0) -- +(-.4,.4);
		\draw[thick] (19.7,0) --+ (-.4,-.4);
		
		\draw[thick] (14.3,0) -- +(.4,.4);
		\draw[thick] (14.3,0) --+ (.4,-.4);
		\draw (16,-3.5) node{$\Gamma_{163}$};
	\end{scope}
	\begin{scope}[shift = {(20,0)}]
		\draw (0,0.2) node{$\stackrel{}{\npreceq}$};
		\draw (13,-3.5) node{$\Gamma_{219}$};
		\draw (5,0) node[below] {} circle (.3cm);
		\draw (8,0) node[below] {} circle (.3cm);
		\draw (11,0) node[below] {} circle (.3cm);
		\draw (14,0) node[below] {} circle (.3cm);
		\draw (17,0) node[below] {} circle (.3cm);
		\draw (20,0) node[below] {} circle (.3cm);

		\draw (5.3,0) -- +(2.4,0);
		\draw (8.3,0) -- +(2.4,0);
		\draw (11.3,0) -- +(2.4,0);
		\draw (14.3,0.15) -- +(2.25,0);
		\draw (14.3,-0.15) -- +(2.25,0);
		\draw (17.3,0) -- +(2.4,0);

		\draw[thick] (16.7,0) -- +(-.4,.4);
		\draw[thick] (16.7,0) --+ (-.4,-.4);

	\end{scope}

	\begin{scope}[shift = {(45,0)}]
		\draw (16,-3.5) node{$\Gamma_{162}$};
		\draw (11,0) node[below] {} circle (.3cm);
		\draw (14,0) node[below] {} circle (.3cm);
		\draw (17,0) node[below] {} circle (.3cm);
		\draw (20,0) node[below] {} circle (.3cm);
		\draw (11.3,0) -- +(2.4,0);
		\draw (14.3,0.15) -- +(2.25,0);
		\draw (14.3,-0.15) -- +(2.25,0);
		\draw (17.5,0.15) -- +(2.25,0);
		\draw (17.5,-0.15) -- +(2.25,0);
		
		\draw[thick] (17.3,0) -- +(.4,.4);
		\draw[thick] (17.3,0) --+ (.4,-.4);
		
			\draw[thick] (16.7,0) -- +(-.4,.4);
		\draw[thick] (16.7,0) --+ (-.4,-.4);

	\end{scope}

	\begin{scope}[shift = {(69,0)}]
		\draw (0,0.3) node{$\stackrel{}{\npreceq}$};
		\draw (13,-3.5) node{$\Gamma_{218}$};
		  \draw (5,0) node[below] {} circle (.3cm);
		\draw (8,0) node[below] {} circle (.3cm);
		\draw (11,0) node[below] {} circle (.3cm);
		\draw (14,0) node[below] {} circle (.3cm);
		\draw (17,0) node[below] {} circle (.3cm);
		\draw (20,0) node[below] {} circle (.3cm);

		  \draw (5.3,0) -- +(2.4,0);
		\draw (8.3,0) -- +(2.4,0);
		\draw (11.3,0) -- +(2.4,0);
		\draw (14.5,0.15) -- +(2.25,0);
		\draw (14.5,-0.15) -- +(2.25,0);
		\draw (17.3,0) -- +(2.4,0);
		
		\draw[thick] (14.3,0) -- +(.4,.4);
		\draw[thick] (14.3,0) --+ (.4,-.4);

	\end{scope}

\end{tikzpicture}

\subsection{Example:} Let $X$ be a doubly-laced Dynkin diagram of a symmetrizable Kac-Moody algebra. Suppose that $\Xsh$ is of type $A_2$, then there is no \pisystem of type $HA_1^{(1)}$ in $X$ contained wholly in $\rerootsh(X)$. The $GCM$ corresponding to $HA_1^{(1)}$ is same as that of $A_2 \times A_1$ (mod 2). So by Corollary \ref*{cor:da2}, there is no \pisystem of type $HA_1^{(1)}$ in $X_{short}$. However in contrast, $A_1^{(1)}$ can appear inside $X_{short}$ when $X=HF_4^{(1)}$. Likewise the rank 2 hyperbolic with the $GCM$ \[M =\begin{pmatrix} 2 & -3 \\ -3 &2 \end{pmatrix}\] cannot occur in cases mentioned in Corollaries \ref*{cor:da1} or \ref{cor:ta1} - since the $GCM$ is same as that of $A_2$ (mod 2) and $A_1 \times A_1$ (mod 3).

\subsection{} In addition to the general assertions of Corollaries \ref{cor:da1}, \ref{cor:da2} and \ref{cor:ta1},  we also have the following two lemmas that only apply when the ambient Lie algebra is of finite, affine or  hyperbolic type.
\begin{lemma}\label{lem:ta1-hyp}
  Suppose $X$ is a triply-laced Dynkin diagram of finite, affine or hyperbolic type. Suppose $\Xsh$ is of type $A_1$, then there is no \pisystem of type $A_2$ in $X$ contained wholly in $\rerootsh(X)$.
\end{lemma}
\begin{proof}
  Let $p$ denote the vertex of $X$ such that $\Xsh=\{p\}$. We normalize the standard invariant form on $X$ such that $|\alpha_p|^2 = \displaystyle\min_{j \in X} |\alpha_j|^2=1$. Since $X$ is triply-laced, $|\alpha_j|^2$ is a nonzero power of $3$ for all $j \neq p$. Now suppose $\Sigma=\{\beta_1, \beta_2\}$ is a \pisystem of type $A_2$ in $X$ such that $\Sigma \subset \rerootsh(X) = W\alpha_p$ (the Weyl group orbit of $\alpha_p$). Applying an element of $W$ if necessary, we can assume $\beta_1=\alpha_p$.
By the arguments used in the proof of Theorem~\ref{thm:pisys-mod-d}, specifically equation~\eqref{eq:imp-obs}, we obtain:
\[ \beta_2=\pm\alpha_p+\gamma\]
for some $\gamma\in R$, where $R =  \integers \,(3\alpha_p) \oplus \bigoplus_{j \neq p} \integers \alpha_j $. Thus
\[\langle\beta_2, \beta_1^\vee\rangle=\pm\langle \alpha_p, \alpha_p^\vee\rangle + \langle \gamma, \alpha_p^\vee \rangle \in \pm 2 + 3\,\mathbb{Z}\]
since $\langle \alpha_j, \alpha_p^\vee\rangle =0$ or $-3$ for all $j \neq p$.
Now $\Sigma$ has type $A_2$, so $\langle\beta_2, \beta_1^\vee\rangle =-1$.
We must have $\beta_2=\alpha_p+\gamma$, with $\langle\gamma, \alpha_p^\vee\rangle=-3$.
We compute:
\[|\beta_2|^2 =|\alpha_p|^2 + |\gamma|^2 + 2\form{\alpha_p}{\gamma} = |\alpha_p|^2 + |\gamma|^2 + \aform{\gamma}{\alpha_p^\vee}\]
since $|\alpha_p|^2=1$. Since $\beta_2$ is $W$-conjugate to $\alpha_p$, their norms coincide, and we obtain $|\gamma|^2=-\aform{\gamma}{\alpha_p^\vee}=3$. We write
\[\gamma=3k_p\,\alpha_p+\sum\limits_{j\neq p} k_j\,\alpha_j\]
where the $k_{\bullet}$ are integers. We observe that $\frac{3k_p|\alpha_p|^2}{|\gamma|^2} = k_p\in\mathbb{Z}$.
For $j \neq p$, $\frac{k_j|\alpha_j|^2}{|\gamma|^2} =  \frac{k_j|\alpha_j|^2}{3}\in\mathbb{Z}$ since $3$ divides $|\alpha_j|^2$. Since $X$ is of finite, affine or hyperbolic type, we use equation~\eqref{eq:reroot-hyp} to conclude that $\gamma$ is a real root of $X$. But $\gamma= \beta_2 - \beta_1$, which contradicts the fact that $\Sigma$ is a \pisystem.
\end{proof}

\begin{lemma} \label{lem:da2-hyp}
Suppose $X$ is a doubly-laced Dynkin diagram of finite, affine or hyperbolic type. Suppose $\Xsh$ is of type $A_1$ or $A_2$, then there is no \pisystem of type $A_1 \times A_1$ in $X$ contained wholly in $\rerootsh(X)$.
\end{lemma}
\begin{proof}
We prove it in the case that $\Xsh$ is of type $A_2$, the other case being similar. So, let $\Xsh=\{p,q\}$ and let $\{\beta_1,\beta_2\}$ be two elements in the $W$-orbit of $\{\alpha_p,\alpha_q\}$ which form a \pisystem of type $A_1 \times A_1$.   Applying an element of $W$ and interchanging $p,q$ if necessary, we can assume $\beta_1=\alpha_p$. By the arguments used in the proof of Theorem~\ref{thm:pisys-mod-d}, we obtain:
  \[ \beta_2= \alpha +\gamma\]
for some $\alpha \in \reroots(\Xsh)$ and $\gamma\in R$ where $R =  2\,Q(\Xsh) \oplus Q(X\backslash \Xsh)$.
We have
\[ 0 = \aform{\beta_2}{\beta^\vee_1} = \aform{\alpha}{\alpha_p^\vee} + \aform{\gamma}{\alpha_p^\vee} \in  \aform{\alpha}{\alpha_p^\vee} + 2 \integers\]
As in the proof of Corollary~\ref{cor:da2}, we note that $\aform{\alpha}{\alpha_p^\vee}$ is even if and only if $\alpha = \pm \alpha_p$. Since $\alpha_p \equiv -\alpha_p \pmod{R}$, we may assume $\beta_2 = \alpha_p + \gamma$. We conclude $ \aform{\gamma}{\alpha_p^\vee} = -2$. Normalizing the standard invariant form such that $|\alpha_p|^2 = |\alpha_q|^2 =1$, we compute: $|\beta_2|^2 = |\alpha_p|^2 + |\gamma|^2 + \aform{\gamma}{\alpha_p^\vee}$. As before, this implies $|\gamma|^2 = -\aform{\gamma}{\alpha_p^\vee} = 2$. Letting:
\[\gamma=2k_p\,\alpha_p+2k_q\,\alpha_q+\sum\limits_{j\neq p,q}k_j\,\alpha_j\]
we obtain: (i) $\frac{2k_p|\alpha_p|^2}{|\gamma|^2} =k_p \in \integers$, (ii) $\frac{2k_q|\alpha_q|^2}{|\gamma|^2} = k_q \in \integers$, and (iii) $\frac{k_j|\alpha_j|^2}{|\gamma|^2} = \frac{k_j|\alpha_j|^2}{2}\in\mathbb{Z}$ for each $j\neq p,q$, since in this case $|\alpha_j|^2$ is a nonzero power of 2. Equation~\eqref{eq:reroot-hyp} implies $\gamma$ is a real root of $X$, contradicting the fact that $\{\beta_1, \beta_2\}$ was a \pisystem to begin with.
\end{proof}

\subsection{} \label{sec:rem-sh-long}
We note that neither of the above lemmas holds if `short' is replaced by `long'. For example:
\be
\item If $X=G_2$, then $\Xlong$ is of type $A_1$. But the set of all long roots forms a {\em closed subroot system} isomorphic to $A_2$; a \pisystem of type $A_2$ in $G_2$ consisting entirely of long roots is $\{\alpha_1, \alpha_1 + 3\alpha_2\}$ where $\alpha_1, \alpha_2$ are respectively the long and short simple roots of $G_2$.
\smallskip\item If $X=B_3$, then $\Xlong =\{p,q\}$ (say) is of type $A_2$. Consider $\Sigma = \{-\theta\} \cup \{\alpha_p, \alpha_q\}$ where $\theta$ is the highest root of $X$. This forms a \pisystem consisting entirely of long roots; it has type $A_3$, and hence contains a subsystem of type $A_1 \times A_1$. 
\ee

\section{\bf{Sufficient conditions for $B \preceq A$: explicit constructions}}\label{suff}

In this section we will develop some principles for explicitly constructing \pisystems in a given Dynkin diagram. These are generalizations of the principles developed in \cite{svis-e10} for simply-laced diagrams. While these principles are widely applicable, we will demonstrate them by constructing many examples of \pisystems of hyperbolic type in hyperbolic Dynkin diagrams.

\subsection{}\label{sec:principles}
All our principles below are instances of the following simple, but powerful method of constructing \pisystems.

{\bf General principle:} Let $X$ be the Dynkin diagram of a symmetrizable GCM. Let $\Lambda$ denote a proper subdiagram of $X$ and let $\Lambda'$ be the subdiagram formed by the vertices not in $\Lambda$. Let $\Sigma, \Sigma'$ be \pisystems in $\Lambda, \Lambda'$ respectively, {\em consisting of positive real roots}. Then $\Sigma \cup \Sigma'$ is a \pisystem in $X$. 

This principle follows from the observations that (i) the (real) roots of a subdiagram are precisely the (real) roots of the ambient diagram that are supported on the subdiagram, (ii) the difference of two positive roots with disjoint supports will have coefficients of mixed sign, and can therefore not be a root.
In all our applications below, we will always take $\Sigma'$ to consist of the set of all simple roots of $\Lambda'$.

Observe that the GCM of $\Sigma \cup \Sigma'$ is of the form
\begin{equation} \label{eq:gcmb}
  \begin{bmatrix} B & * \\ * & B' \end{bmatrix}
\end{equation}
where $B, B'$ are the GCMs of $\Sigma, \Sigma'$ respectively. The terms denoted $*$ are of the form $2{\form{\beta_1}{\beta_2}}/{\form{\beta_2}{\beta_2}}$ where $\beta_1 \in \Lambda, \beta_2 \in \Lambda'$ or vice versa.  We now isolate some special instances of this general principle, which will be used repeatedly in the sequel.

\subsection{Principle A:} Let $Y$ be an affine Dynkin diagram, twisted or untwisted, but $Y \neq A_{2l}^{(2)}$. Let $\{\alpha_0,\cdots,\alpha_n\}$ denote the simple roots of $Y$. Let $\overline{Y}$ denote the underlying finite type diagram, obtained from $Y$ by deleting the node corresponding to $\alpha_0$.

Let $X$ be the diagram obtained by adding an extra vertex to $Y$, which is connected only to $\alpha_0$, and by a single edge. Since $Y$ is symmetrizable, so is $X$. We denote the simple root corresponding to this vertex $\alpha_{-1}$.
Let $A=(a_{ij})$ denote the GCM of $X$; thus $a_{ij} = {2\form{\alpha_i}{\alpha_j}}/{\form{\alpha_i}{\alpha_i}}$ for $-1 \leq i,j \leq n$
(we note in passing that when $Y$ is simply-laced, $X$ is of $\Ext$ type). Let $\delta_Y$ denote the null root of $Y$, so $\delta_Y=\sum_{i=0}^{n} a_i\alpha_i$ with $a_i \in \mathbb{N}$.
We let $s_i$ denote the reflection corresponding to the simple root $\alpha_i$.

Since $Y$ is an affine diagram other than $A_{2l}^{(2)}$, we have $a_0=1$ \cite[Chapter 4, Tables Aff 1-3]{kac}. In the general principle, we take the subdiagram $\Lambda=Y$ and $\Lambda'$ to be the singleton set containing the vertex $(-1)$.  Define $\Sigma$ to be the \pisystem in $Y$ of type $Y$ comprising the roots $\{s_0\,\gamma_i: 0 \leq i \leq n\}$ where the $\gamma_i$ are given by:
\[ \gamma_0 = \alpha_0+\delta_Y,\;\; \gamma_j = \al_j \; (j\geq 1)\]
When $Y$ is twisted, $\alpha_0$ is a short root and hence $\alpha_0 + \delta_Y$ is a root in this case; it is of course a root when $Y$ is untwisted. 
Define $\Sigma' = \{\alpha_{-1}\}$; this is clearly of finite type $A_1$.
We let $\Sigma \cup \Sigma' = \{\beta_i: -1 \leq i \leq n\}$ with $\beta_{-1} = \alpha_{-1}$ and $\beta_i = s_0 \, \gamma_i$ for $i \geq 1$.
All the hypotheses of the general principle are satisfied. As observed in equation~\eqref{eq:gcmb}, to find the type of $\Sigma \cup \Sigma'$, it only remains to compute the numbers
$b_{ij} = {2\form{\beta_i}{\beta_j}}/{\form{\beta_i}{\beta_i}}$ where 
$i=-1, j \geq 0$ or vice-versa.

Now: (i) $\form{\beta_{-1}}{\beta_j} = \form{s_0\,\beta_{-1}}{\gamma_j} = \form{\alpha_0}{\alpha_j}$ for $j \geq 1$, since $s_0\,\alpha_{-1} = \alpha_0 + \alpha_{-1}$ and $\alpha_{-1}$ is orthogonal to all roots of $\overline{Y}$. (ii) $|\beta_{-1}|^2 = |\alpha_{-1}|^2 = |\alpha_0|^2$. This gives us: $b_{j,-1} = a_{j0}$ and $b_{-1,j} = a_{0j}$ for $j \geq 1$.
Finally, we compute: $\form{\beta_{-1}}{\beta_0} = \form{\alpha_{-1}}{s_0(\alpha_0 + \delta_Y)}$. But $s_0(\alpha_0 + \delta_Y) = -\alpha_0 + \delta_Y = \theta$,  where $\theta$ is the highest long (respectively short) root of $\overline{Y}$ if $Y$ is untwisted (respectively twisted). But $\form{\alpha_{-1}}{\theta} =0$ since as before $\alpha_{-1}$ is orthogonal to all roots of $\overline{Y}$. In other words $b_{0,-1} =b_{-1,0} =0$.

The Dynkin diagram $S(B)$ is thus obtained from $X = S(A)$ by removing the edge between vertices $0$ and $-1$, and instead connecting the vertex $-1$ to every neighbour of $0$ with the same edge labels, i.e., such that   $b_{j,-1} = a_{j0}$ and $b_{-1,j} = a_{0j}$.

\subsection{Principle B:} Let $X$ be the Dynkin diagram of a symmetrizable GCM $A$ and let $Y$ denote a subset of its vertices such that $Y$ forms a subdiagram of affine type. We set $r=1$ if $Y$ is untwisted, $r=3$ if $Y$ if of type $D_4^{(3)}$ and $r=2$ for all other twisted types. Let $\delta_Y$ denote the null root of the diagram $Y$. In the general principle, we choose $\Lambda = Y$. For each $p \in Y$, fix a non-negative integer $k_p$; if $\alpha_p$ is a long root of $Y$, we require further that $r | k_p$ (for $Y$ of type $A_{2n}^{(2)}$ this requirement only applies to the longest root length). Let $\beta_p = \alpha_p + k_p \, \delta_Y$ and define
$\Sigma =\{\beta_p: p \in Y\}$; this is a \pisystem of type $Y$ in $Y$. For $q \not\in Y$, let $\beta_q = \alpha_q$ and define $\Sigma' =\{\beta_q: q \not\in Y\}$. Then, by the general principle, $\Sigma \cup \Sigma'$ is a \pisystem in $X$. Let $B=(b_{ij})_{i,j \in X}$ denote its type. As above, $b_{ij} = a_{ij}$ whenever $i,j$ are both in $Y$ or both not in $Y$. To compute $b_{pq}$ and $b_{qp}$ for $p \in Y, q \not\in Y$, we have:
\[\form{\beta_p}{\beta_q} = \form{\al_p}{\alpha_q} + k_p\form{\delta_Y}{\alpha_q}\]
Hence $b_{pq} = a_{pq} + k_p\frac{2\form{\delta_Y}{\alpha_q}}{\form{\alpha_p}{\alpha_p}}$ and $b_{qp} = a_{qp} + k_p\frac{2\form{\delta_Y}{\alpha_q}}{\form{\alpha_q}{\alpha_q}}$.
These can be explicitly computed in each case of interest.

While we will have occassion to use this principle in its full generality, we give below some special instances of it which occur often. Since $Y$ is affine, we assume that the vertices of $Y$ have the standard labelling $0,1, \cdots, n$ as in \cite[Chapter 4]{kac}. Suppose $X\backslash Y$ contains only a single vertex (labelled $-1$) which is connected by a single edge to the vertex $0$ of $Y$.  

(i) First let us suppose that $Y$ is untwisted. Fix $p$ such that $1 \leq p \leq n$. Choose $k_p=1$ and $k_s=0$ for all $0 \leq s \leq n$, $s \neq p$.
We only need to compute $b_{ij}$ for $i=-1, j \geq 0$ or vice-versa. Now, clearly $b_{-1,j} = a_{-1,j}$ and $b_{j,-1}=a_{j,-1}$ for $j \geq 0$, $j \neq p$. Further,
\[\form{\beta_p}{\beta_{-1}} = \form{\al_p}{\alpha_{-1}} + \form{\delta_Y}{\alpha_{-1}} = \form{\alpha_0}{\alpha_{-1}} = -\frac{|\alpha_0|^2}{2}\]

Since $|\beta_i|^2=|\alpha_i|^2$ for all $i$, we conclude that  $b_{-1,p} =-|\alpha_0|^2/|\alpha_{-1}|^2 = -1$ and $b_{p,-1} =-|\alpha_0|^2/|\alpha_p|^2$. Now since $\alpha_0$ is a long root of $Y$, we obtain
\[b_{p,-1} =\begin{cases} -1 & \text{ if } \alpha_p \text{ is a long root of } Y \\
-2 & \text{ if } Y \neq G_2^{(1)}, \text{ and } \alpha_p \text{ is a short root of } Y \\
-3 & \text{ if } Y=G_2^{(1)}, \text{ and } \alpha_p \text{ is a short root of } Y
\end{cases}
\]
In terms of Dynkin diagrams, the diagram $S(B)$ coincides with $S(A)$ except that there is a single, double or triple edge joining vertices $-1$ and $p$ (with an arrow pointing towards $p$) depending on the three cases above.

\smallskip
(ii) If $Y$ is twisted, fix a vertex $1 \leq p \leq n$ and define (i) $k_s =0$ for $0 \leq s \leq n$, $s \neq p$ (ii) $k_p =r$ if $\alpha_p$ is a long root (longest root in case of $A_{2n}^{(2)}$) and $k_p=1$ otherwise. As above we have: (a) $b_{ij} = a_{ij}$ for $i,j \neq p$, (b) $b_{ij} = a_{ij}$ for $i,j \neq -1$, (c) $b_{p,-1}=-1$ and (d) $b_{-1,p} = -|\alpha_p|^2/|\alpha_0|^2$. 
Since $\alpha_0$ is a short root of $Y$, we have:

\[b_{-1,p} =\begin{cases} -1 & \text{ if } \alpha_p \text{ is not a long root of } Y \\
-2 & \text{ if } Y \neq D_4^{(3)}, \text{ and } \alpha_p \text{ is a long root of } Y \\
-3 & \text{ if } Y=D_4^{(3)}, \text{ and } \alpha_p \text{ is a long root of } Y
\end{cases}
\]
As before, this implies that the diagram $S(B)$ coincides with $S(A)$ except that there is a single, double or triple edge joining vertices $-1$ and $p$ (with an arrow pointing away from $p$) depending on the three cases above.

\smallskip
(iii) If instead of $1 \leq p \leq n$, we choose the vertex $p=0$ in (i) or (ii) above, we obtain $b_{0,-1} = b_{-1,0} = -2$, and $b_{ij} = a_{ij}$ for all other pairs $(i,j)$. In the Dynkin diagram $S(B)$, this would be denoted by a double edge between vertices $0$ and $-1$, marked with two arrows, one pointing toward each vertex.

\subsection{}
For principles {\bf C, D, E}, we let $X$ denote the Dynkin diagram of any symmetrizable GCM.

\smallskip
\noindent
    {\bf Principle C:} (Shrinking)  Suppose $I$ is a subset of the vertices of $X$ such that $I$ forms a (connected) subdiagram of Finite type. It is well known that  $\beta_{\bullet} = \sum_{i \in I} \al_i$ is a root of $\lieg(I)$. Since $I$ is of finite type, this root is real. In the general principle, we choose the subset $\Lambda = I$ and the \pisystem $\Sigma = \{\beta_{\bullet}\}$. Let $\Sigma'= \{\al_j: j \not\in I\}$.
Let $B$ denote the GCM of $\Sigma \cup \Sigma'$. We have for $j \not\in I$,
$$\frac{\form{\beta_{\bullet}}{\alpha_j}}{\form{\alpha_j}{\alpha_j}} = \sum_{i\in I} \frac{\form{\alpha_i}{\alpha_j}}{\form{\alpha_j}{\alpha_j}}$$ 
Further, letting $k_i = |\alpha_i|^2/|\beta_{\bullet}|^2 $ for  $i \in I$, we have
$$\frac{\form{\beta_{\bullet}}{\alpha_j}}{\form{\beta_{\bullet}}{\beta_{\bullet}}} = \sum_{i \in I} k_i\frac{\form{\alpha_i}{\alpha_j}}{\form{\alpha_i}{\alpha_i}}$$ 
Thus,
\begin{equation} \label{eq:bullet-weights}
  b_{j\bullet}=\sum_{i \in I} a_{ji}, \;\;\; b_{\bullet j}=\sum_{i \in I} k_i\,a_{ij}
  \end{equation}

We note that $k_i$ is the ratio of root lengths in a finite type diagram, and is therefore one of $\frac{1}{3}, \frac{1}{2}, 1, 2, 3$.
If no two vertices of $I$ have a common neighbour $j \not\in I$, then the Dynkin diagram $S(B)$ may be thought of as being obtained from $X$ by contracting the vertices of $I$ to a single ``fat'' vertex $\bullet$. The edges in $X$ between $i\in I$ and $j \not\in I$ are now drawn between $\bullet$ and $j$ in $S(B)$ (with possibly new edge weights). The rest of the diagram $X$ is carried over unchanged.

\medskip
\noindent
    {\bf Principle D:} (Deletion) If we delete any subset of vertices from the vertex set of $X$ and define $\Sigma$ to be the set of remaining $\{\alpha_i\}$, then $\Sigma$ is a \pisystem in $X$. Its Dynkin diagram is clearly a subdiagram of $X$.

\medskip
\noindent
    {\bf Principle E:}

    (i)  Let the vertices of $X$ be labelled $1, 2, \cdots, n$. Suppose $X$ contains a subdiagram of finite type $B_2$, i.e., there are vertices $p,q$ in $X$ joined by a double bond directed (say) towards $p$. In other words, $a_{pq}=-2, a_{qp}=-1$. In the general principle, we take $\Lambda$ to be this subdiagram of type $B_2$ and define $\Sigma =\{\beta_p, \beta_q\}$ to be the \pisystem of type $A_1 \times A_1$ in $\Lambda$ given by:
    \[ \beta_p =s_p(\alpha_{q})=\alpha_{q}+2\alpha_p, \;\;\;\; \beta_q = \alpha_q.\]
    
    Define $\beta_j = \alpha_j$ for $1 \leq j \leq n, \,j\neq p,q$ and let $\Sigma'$ be the set of these $\beta_j$.
    Let $B$ denote the GCM of $\Sigma \cup \Sigma' = \{\beta_i: 1 \leq i \leq n\}$; clearly $b_{ij} = a_{ij}$ for $i,j \neq p$.
Now,
\[\frac{\form{\beta_p}{\beta_j}}{\form{\beta_j}{\beta_j}} = \frac{\form{\alpha_{q}}{\alpha_j}}{\form{\alpha_j}{\alpha_j}} + 2\frac{\form{\alpha_p}{\alpha_j}}{\form{\alpha_j}{\alpha_j}}, \text{ i.e., }
 b_{jp}=a_{jq} + 2a_{jp}\]
Since $|\alpha_{q}|^2 = 2|\alpha_{p}|^2$, we have
\[\frac{\form{\beta_p}{\beta_j}}{\form{\beta_p}{\beta_p}} = \frac{\form{\alpha_{q}}{\alpha_j}}{\form{\alpha_{q}}{\alpha_{q}}} + 2\frac{\form{\alpha_p}{\alpha_j}}{2\form{\alpha_p}{\alpha_p}} \text{ i.e, }
b_{pj}=a_{qj} + a_{pj}\]
Note in particular that since $\Sigma$ has type $A_1 \times A_1$, we have $b_{pq} = b_{qp} =0$, i.e., the double edge between $p,q$ in $X$ has been removed in $S(B)$.

\smallskip
\noindent
(ii) Now suppose the Dynkin diagram $X$ has a subdiagram of finite type $G_2$, i.e.,  there are vertices $p,q$ in $X$ joined by a triple bond directed towards $p$.
As above, choose $\Lambda$ to be this subdiagram of type $G_2$ and define $\Sigma =\{\beta_p, \beta_q\}$ to be the \pisystem of type $A_2$ in $\Lambda$ given by:
$$ \beta_p = s_p(\alpha_{q})=\alpha_{q}+3\alpha_p, \;\;\; \beta_q = \alpha_q.$$
Choose $\Sigma'$ as above, to consist of all the simple roots $\alpha_i$ of $X$ other than $i=p,q$. A similar computation establishes that $b_{jp}=a_{jq} + 3a_{jp}, \; b_{pj}=a_{qj} + a_{pj}$ and $b_{ij} = a_{ij}$ for all other pairs  $(i,j)$.
Note in particular that since $\Sigma$ is of type $A_2$, one has $b_{pq} = b_{qp} =-1$, i.e., the triple edge between $p,q$ in $X$ has now been replaced by a single edge in $S(B)$.

\smallskip
\noindent
    (iii)  Suppose $X$ contains a subdiagram of type $A_2^{(2)}$, i.e., there are vertices $p,q$ in $X$ with $a_{pq}=-4, a_{qp}=-1$ (depicted in the Dynkin diagram by four bonds directed towards $p$). We choose $\Sigma =\{\beta_p, \beta_q\}$ to be the \pisystem of type $A_1^{(1)}$ in $\Lambda$ given by:
    \[ \beta_p =s_p(\alpha_{q})=\alpha_{q}+4\alpha_p, \;\;\;\; \beta_q = \alpha_q.\]
Reasoning as before, we deduce $b_{jp}=a_{jq} + 4a_{jp}, \; b_{pj}=a_{qj} + a_{pj}$ and $b_{ij} = a_{ij}$ for all other pairs  $(i,j)$.
Here, since $\Sigma$ has type $A_1^{(1)}$, the quadruple edge from $q$ to $p$ has been replaced by a two-way double edge.

\medskip

\medskip
\section{\bf{Non-Maximal Hyperbolic Diagrams}}\label{sec:nonmax}
As an interesting application of the principles of the previous section, we now explicitly demonstrate how to construct linearly independent \pisystems of type $B$ in $A$ for various pairs $(A, B)$ of hyperbolic \pisystems.

\subsection{} In Tables I, II, III 
	 we have listed all the 142 symmetrizable hyperbolic Dynkin diagrams in ranks 3-10. We will denote by $\Gamma_k$ the hyperbolic Dynkin diagram occurring with serial number $k$ in these tables. These diagrams are taken from Tables 1--23 of \cite{lisa-et-al} which contain the full list of 238 hyperbolic diagrams without the assumption of symmetrizability. The diagram $\Gamma_k$ occurs as item number $k$ in Tables 1--23 of \cite{lisa-et-al}. Since we only consider the 142 symmetrizable hyperbolic diagrams rather than all 238 of them, there are ``gaps'' in the serial numbers that occur in our tables.

The entries in our tables contain the following information: for each serial number $k$, the second column is the corresponding Dynkin diagram, the third column is another serial number, say $\ell$ such that $\Gamma_k \preceq \Gamma_\ell$ and the fourth column indicates the principle(s) used to construct a \pisystem of type $\Gamma_k$ in $\Gamma_\ell$. We note that $\ell$ is not unique in general, but since our primary goal is to identify the maximal diagrams relative to $\preceq$, we will be content with finding one value of $\ell$.

The diagrams $\Gamma_k$ for which we are unable to find a suitable $\ell$ using any of our principles are candidates for maximal elements. We show in \S\ref{sec:max-proof} that each of these diagrams is indeed maximal. The entries corresponding to these diagrams are indicated by `Max' in the third column while the fourth column contains the value of the determinant of the GCM of the diagram.
 
In this section we give a few examples to illustrate the Principles A-E developed in the previous section. The other entries of the table may be verified by similar arguments.

\medskip
    
{\bf Principle A:} Taking $X = \Gamma_{219}$ and $Y = F_4^{(1)}$ in principle $A$, we obtain a \pisystem of type $\Gamma_{207}$ in $\Gamma_{219}$. Similarly, choosing $X = \Gamma_{159}$ and $Y = G_2^{(1)}$, we obtain  $\Gamma_{150} \preceq \Gamma_{159}$.

\medskip

{\bf Principle B:} Let $X = \Gamma_{159}$,  $Y=G_2^{(1)}$ and $\alpha_p$ be the long simple root of $G_2$. Applying principle $B$ allows us to construct a \pisystem of type $\Gamma_{129}$ in $\Gamma_{159}$. Similarly, taking $X = \Gamma_{160}$, $Y$ to be the twisted affine diagram $D_4^{(3)}$ and $\alpha_p$ to be the short simple root of $G_2$, we conclude that $\Gamma_{130} \preceq \Gamma_{160}$.

\medskip

{\bf Principle C:} Principle $C$ allows us to shrink diagrams in a specified manner. For instance, one readily obtains from this principle that:
$\Gamma_{222} \preceq \Gamma_{226} \preceq \Gamma_{231} \preceq \Gamma_{236}$.

\medskip
{\bf Principle D:} Typically the deletion principle $D$ is used in conjunction with one of the other principles. For instance, first applying principle $B$ to $X=\Gamma_{163}$, $Y=D_3^{(2)}$ and $p=0$ (i.e., the affine simple root of $Y$) one obtains the rank 4 diagram obtained from $\Gamma_{163}$ by replacing its single edge by the two-way double edge $\Longleftrightarrow$. Now applying principle $D$ to delete the node at the other end gives us $\Gamma_{106}$.

\medskip
    {\bf Principle E:} This principle only applies when the ambient diagram has a double, triple or quadruple edge. For example, an application of this principle shows $\Gamma_{220} \preceq \Gamma_{218}$, $\Gamma_{161} \preceq \Gamma_{160}$ and $\Gamma_{90} \preceq \Gamma_{123}$.

We close this subsection with the example of $\Gamma_{223} \succeq \Gamma_{212}$  which requires a sequential application of the three principles $B, C$ and $E$:

\medskip
\noindent
\begin{tikzpicture}[scale=.19]

  \begin{scope}[shift={(-5,0)}]
  \draw (5,0) node[below] {} circle (.3cm);
  \draw (8,0) node[below] {} circle (.3cm);
  \draw (11,0) node[below] {} circle (.3cm);
  \draw (14,0) node[below] {} circle (.3cm);
  \draw (17,0) node[below] {} circle (.3cm);
  \draw (20,0) node[below] {} circle (.3cm);
  \draw (11,3) node[below] {} circle (.3cm);
  
  \draw (5.3,0) -- +(2.4,0);
  \draw (8.3,0) -- +(2.4,0);
  \draw (11.3,0) -- +(2.4,0);
  \draw (11,0.3) -- +(0,2.4);
  \draw (14.3,0) -- +(2.4,0);
  \draw (17.5,0.15) -- +(2.25,0);
  \draw (17.5,-0.15) -- +(2.25,0);
  
  \draw[thick] (17.3,0) -- +(.4,.4);
  \draw[thick] (17.3,0) --+ (.4,-.4);
\end{scope}
\begin{scope}[shift = {(19,0)}]
  \draw (0,1) node{$\stackrel{B}{\succeq}$};
  \draw (5,0) node[below] {} circle (.3cm);
  \draw (8,0) node[below] {} circle (.3cm);
  \draw (11,0) node[below] {} circle (.3cm);
  \draw (14,0) node[below] {} circle (.3cm);
  \draw (17,0) node[below] {} circle (.3cm);
  \draw (20,0) node[below] {} circle (.3cm);
  \draw (11,3) node[below] {} circle (.3cm);
  
  \draw (5.3,0) -- +(2.4,0);
  \draw (8.3,0) -- +(2.4,0);
  \draw (11.3,0) -- +(2.4,0);
  \draw (11,0.3) -- +(0,2.4);
  \draw (14.3,0) -- +(2.4,0);
  \draw (17.5,0.15) -- +(2.25,0);
  \draw (17.5,-0.15) -- +(2.25,0);
  
  \draw[thick] (17.3,0) -- +(.4,.4);
  \draw[thick] (17.3,0) --+ (.4,-.4);

  \draw    (5.25,-0.15) to[out=-30,in=-150] (19.75,-0.15);
  \draw    (5,-0.3) to[out=-30,in=-150] (20,-.3);
  
  \draw[thick] (12.3,-2.4) -- +(.4,.4);
  \draw[thick] (12.3,-2.4) --+ (.4,-.4);

\end{scope}

\begin{scope}[shift = {(43,0)}]
    \draw (0,1) node{$\stackrel{C}{\succeq}$};
  \draw (5,0) node[below] {} circle (.3cm);
  \draw (11,0) node[below] {} circle (.3cm);
  \draw (14,0) node[below] {} circle (.3cm);
  \draw (17,0) node[below] {} circle (.3cm);
  \draw (20,0) node[below] {} circle (.3cm);
  \draw (11,3) node[below] {} circle (.3cm);
  
  \draw (5.3,0) -- +(5.4,0);
  \draw (11.3,0) -- +(2.4,0);
  \draw (11,0.3) -- +(0,2.4);
  \draw (14.3,0) -- +(2.4,0);
  \draw (17.5,0.15) -- +(2.25,0);
  \draw (17.5,-0.15) -- +(2.25,0);
  
  \draw[thick] (17.3,0) -- +(.4,.4);
  \draw[thick] (17.3,0) --+ (.4,-.4);

  \draw    (5.25,-0.15) to[out=-30,in=-150] (19.75,-0.15);
  \draw    (5,-0.3) to[out=-30,in=-150] (20,-.3);
  
  \draw[thick] (12.3,-2.4) -- +(.4,.4);
  \draw[thick] (12.3,-2.4) --+ (.4,-.4);

\end{scope}

\begin{scope}[shift = {(63,0)}]
    \draw (4,1) node{$\stackrel{E}{\succeq}$};
  \draw (8,0) node[below] {} circle (.3cm);
  \draw (11,0) node[below] {} circle (.3cm);
  \draw (14,0) node[below] {} circle (.3cm);
  \draw (17,0) node[below] {} circle (.3cm);
  \draw (20,0) node[below] {} circle (.3cm);
  \draw (11,3) node[below] {} circle (.3cm);

  \draw (8.3,0) -- +(2.4,0);
  \draw (11.3,0) -- +(2.4,0);
  \draw (11,0.3) -- +(0,2.4);
  \draw (14.5,0.15) -- +(2.25,0);
  \draw (14.5,-0.15) -- +(2.25,0);
  
  \draw[thick] (14.3,0) -- +(.4,.4);
  \draw[thick] (14.3,0) --+ (.4,-.4);

  \draw    (8.25,-0.15) to[out=-30,in=-150] (19.75,-0.15);
  \draw    (8,-0.3) to[out=-30,in=-150] (20,-.3);
  
  \draw[thick] (13.3,-2) -- +(.4,.4);
  \draw[thick] (13.3,-2) --+ (.4,-.4);

\end{scope}

\end{tikzpicture}

\subsection{The exceptions : principle (*)}

As mentioned above, for each non-maximal diagram $\Gamma_k$, Principles A-E can typically be used to exhibit a diagram $\Gamma_\ell$ such that $\Gamma_k \preceq \Gamma_\ell$. However, there are four non-maximal diagrams which are not directly amenable to any of these principles.
We give below special constructions in these cases.

\smallskip
\noindent

(i)  $\Gamma_{91} \preceq \Gamma_{157}$: Consider the Dynkin diagram $\Gamma_{157}$:

\medskip
\begin{center}
 \begin{tikzpicture}[scale=.3]
    \draw (8,0) node[below] {$\alpha_1$} circle (.3cm);
      \draw (11,0) node[below] {$\alpha_2$} circle (.3cm);
      \draw (14,0) node[below] {$\alpha_3$} circle (.3cm);
      \draw (11,3) node[right] {$\alpha_4$} circle (.3cm);

      \draw (8.5,0.15) -- +(2.25,0);
      \draw (8.5,-0.15) -- +(2.25,0);
      
      \draw[thick] (8.3,0) -- +(.4,.4);
      \draw[thick] (8.3,0) --+ (.4,-.4);

      \draw (11.5,0.15) -- +(2.25,0);
      \draw (11.5,-0.15) -- +(2.25,0);
      
      \draw[thick] (11.3,0) -- +(.4,.4);
      \draw[thick] (11.3,0) --+ (.4,-.4);

      \draw(11.15,0.25) -- +(0,2.25);
      \draw(10.85,0.25) -- +(0,2.25);

      \draw[thick] (11,2.7) -- +(-.4,-.4);
      \draw[thick] (11,2.7) --+ (.4,-.4);

  \end{tikzpicture}
\end{center}

The \pisystem $\Sigma=\{\alpha_1 + \alpha_2, \alpha_3, \alpha_1 + \alpha_2 + 2\alpha_4  \}$ is of type $\Gamma_{91}$.

\smallskip
\noindent

(ii)  $\Gamma_{158} \preceq \Gamma_{191}$:  Consider the Dynkin diagram $\Gamma_{191}$:

\medskip
\begin{center}
 \begin{tikzpicture}[scale=.3]
    \draw (8,0) node[below] {$\alpha_1$} circle (.3cm);
      \draw (11,0) node[below] {$\alpha_2$} circle (.3cm);
      \draw (14,0) node[below] {$\alpha_3$} circle (.3cm);
      \draw (17,0) node[below] {$\alpha_4$} circle (.3cm);
      \draw (11,3) node[right] {$\alpha_5$} circle (.3cm);

      \draw (8.2,0.2) -- +(2.3,0);
      \draw (8.2,-0.2) -- +(2.3,0);
      \draw[thick] (10.7,0) -- +(-.4,-.4);
      \draw[thick] (10.7,0) --+ (-.4,.4);

      \draw (11.3,0) -- +(2.4,0);
      \draw (11,0.3) -- +(0,2.4);

      \draw (14.2,0.2) -- +(2.3,0);
      \draw (14.2,-0.2) -- +(2.3,0);
      \draw[thick] (16.7,0) -- +(-.4,-.4);
      \draw[thick] (16.7,0) --+ (-.4,.4);

  \end{tikzpicture}
\end{center}
The \pisystem $\Sigma=\{\alpha_1, \alpha_1 + 2\alpha_2, \alpha_5 + \alpha_2+\alpha_3, \alpha_4\}$ is of type $\Gamma_{158}$.

 \smallskip
 \noindent

 (iii)  $\Gamma_{172} \preceq \Gamma_{160}$: Consider the Dynkin diagram $\Gamma_{160}$:

\medskip

\begin{center}
 
 \begin{tikzpicture}[scale=.3]
    \draw (8,0) node[below] {$\alpha_1$} circle (.3cm);
      \draw (11,0) node[below] {$\alpha_2$} circle (.3cm);
      \draw (14,0) node[below] {$\alpha_3$} circle (.3cm);
      \draw (17,0) node[below] {$\alpha_4$} circle (.3cm);
      
      \draw (8.3,0) -- +(2.4,0);
      \draw (11.3,0) -- +(2.4,0);
      \draw (14.5,0.2) -- +(2.3,0);
      \draw (14.5,-0.2) -- +(2.3,0);
      \draw (14.3,0) -- +(2.4,0);
      
      \draw[thick] (14.3,0) -- +(.4,.4);
      \draw[thick] (14.3,0) --+ (.4,-.4);
  \end{tikzpicture}
\end{center}

The \pisystem $\Sigma=\{\alpha_1+\alpha_2+\alpha_3, \; \alpha_4, \; \alpha_4+3\alpha_3,\;  \alpha_2\}$ is of type $\Gamma_{172}$.

 \smallskip
 \noindent
 (iv)  $\Gamma_{214} \preceq \Gamma_{218}$: Consider the Dynkin diagram $\Gamma_{218}$:

\medskip
\begin{center}
 \begin{tikzpicture}[scale=.3]
    \draw (8,0) node[below] {$\alpha_1$} circle (.3cm);
      \draw (11,0) node[below] {$\alpha_2$} circle (.3cm);
      \draw (14,0) node[below] {$\alpha_3$} circle (.3cm);
      \draw (17,0) node[below] {$\alpha_4$} circle (.3cm);
      \draw (20,0) node[below] {$\alpha_5$} circle (.3cm);
      \draw (23,0) node[below] {$\alpha_6$} circle (.3cm);
      
      \draw (8.3,0) -- +(2.4,0);
      \draw (11.3,0) -- +(2.4,0);
      \draw (14.3,0) -- +(2.4,0);
      \draw (20.3,0) -- +(2.4,0);
      \draw (17.5,0.15) -- +(2.25,0);
      \draw (17.5,-0.15) -- +(2.25,0);
      
      \draw[thick] (17.3,0) -- +(.4,.4);
      \draw[thick] (17.3,0) --+ (.4,-.4);
  \end{tikzpicture}
\end{center}
The \pisystem $\Sigma=\{\alpha_1, \alpha_2, \alpha_5 + 2\alpha_4+2\alpha_3, \alpha_6, \alpha_5, \alpha_4\}$ is of type $\Gamma_{214}$.

\section{Maximal Hyperbolic diagrams}\label{sec:max-proof}
In this section, we consider the 22 symmetrizable hyperbolic diagrams $\Gamma_k$ which cannot be exhibited as \pisystems of other diagrams using Principles A-E. Such diagrams only exist in ranks 3, 4, 6 and 10 and there are 5, 9, 5 and 3 such diagrams (respectively) in those ranks. We will prove that these are all in fact maximal diagrams relative to the partial order $\preceq$. As mentioned in \S\ref{sec:nonmax}, the entries corresponding to these diagrams are labelled `Max' in the third column and contain the determinant of their GCMs in the fourth.

\subsection{Rank 10} Since $\det \Gamma_{238} =-1$, it is maximal by the determinant criterion (lemma~\ref{lem:det-multiple}). The same lemma shows that $\Gamma_{236}$ and $\Gamma_{237}$ are not $\preceq$ comparable. Both these latter diagrams have two root lengths, while $\Gamma_{238}$ has only one, so the root length criterion (lemma~\ref{lem:rlr}) shows that neither of them can be $\preceq \Gamma_{238}$. Thus all three are maximal diagrams of rank 10.

\subsection{Rank 6} Since $\Gamma_{218}$ and $\Gamma_{219}$ have determinant $-1$, they are both maximal among rank 6 diagrams by the determinant criterion. The root length criterion ensures that neither of these is $\preceq \Gamma_{238}$, so to show maximality of these two diagrams, it only remains to prove that neither of them can be realized as \pisystems of $\Gamma_{236}$ or $\Gamma_{237}$. But this follows readily from corollary~\ref{cor:da1}.

Diagrams $\Gamma_{216}$ and $\Gamma_{217}$ have three root lengths. By the root length criterion they cannot be realized as \pisystems of any of the rank 10 maximal diagrams or of the other candidate diagrams $\Gamma_k$ ($k=215, 218, 219$) in rank 6. Since each of these two diagrams have determinant $-2$, they are mutually incomparable by the determinant criterion. This establishes maximality of $\Gamma_{216}$ and $\Gamma_{217}$.

Finally to show maximality of $\Gamma_{215}$, we observe that it cannot be realized as a \pisystem of:
(i) $\Gamma_k$ for $k=236, 237$ by corollary~\ref{cor:da1}
(ii) $\Gamma_{238}$ by the root length criterion
(iii) $\Gamma_k$ for $k=216, 217$ by the determinant criterion
(iv) $\Gamma_{218}$ by corollary~\ref{cor:da2}
(v) $\Gamma_{219}$ by lemma~\ref{lem:da2-hyp}.

\subsection{Rank 4} Since $\det \Gamma_{159} = \det \Gamma_{160} =-1$, they are maximal amongst rank 4 diagrams. Since both these diagrams are triply laced, they contain a pair of simple roots $\alpha_i, \alpha_j$ such that $|\alpha_i|^2/|\alpha_j|^2 = 3$. However none of the maximal diagrams in rank 6 or 10 have triple edges, so the root length criterion ensures that neither of $\Gamma_{159}, \Gamma_{160}$ occur as \pisystems of those diagrams. Hence $\Gamma_{159}$ and $\Gamma_{160}$ are maximal.

The root length criterion shows that $\Gamma_{173}$ is maximal since it contains 4 root lengths. It also shows that none of the $\Gamma_k$ for $166 \leq k \leq 170$ can be realized as \pisystems of $\Gamma_{159}$ or $\Gamma_{160}$ or of any of the maximal diagrams of ranks 6 or 10. Since $\det \Gamma_k = -2$ or $-3$ for $166 \leq k \leq 170$, the determinant criterion implies they are pairwise incomparable. This establishes their maximality.

Finally to show maximality of $\Gamma_{171}$, we observe that it cannot be realized as a \pisystem of:
(i) any of the maximal diagrams of rank 6 or 10, by the root length criterion 
(ii)  $\Gamma_k$ for $166 \leq k \leq 170$, by the determinant criterion
(iii) $\Gamma_{160}$ by corollary~\ref{cor:ta1}
(iv) $\Gamma_{159}$ by lemma~\ref{lem:ta1-hyp}.

\subsection{Rank 3}  The determinant criterion ensures that $\Gamma_k$, $117 \leq k \leq 121$ are pairwise incomparable. By the root length criterion, these diagrams cannot be realized as \pisystems of any diagram of rank $\geq 4$. Thus, they are all maximal.

\subsection{Remarks}\label{sec:rems-dualization} This completes the verification that all 22 candidate diagrams in ranks 3-10 are in fact maximal.

\newpage
\centerline{TABLE I: Hyperbolic diagrams of ranks 3, 4}
\begin{tikzpicture}
  \node at (0,0) {\resizebox{4.25in}{4in}{\includegraphics*[viewport=5 5 720 720 ]{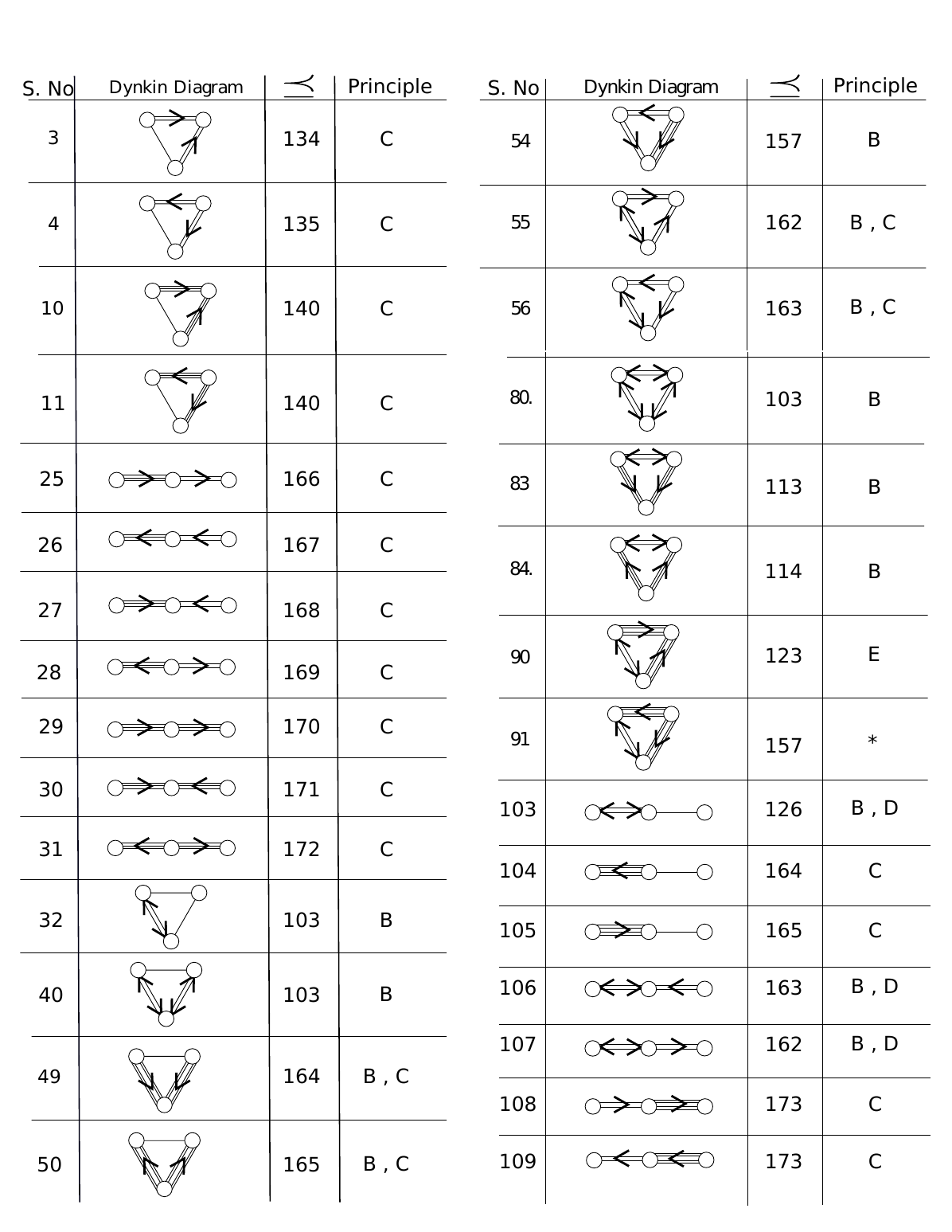}}};
\node at (9,3) {\resizebox{4.25 in}{1.8in}{\includegraphics*[viewport=5 400 720 720 ]{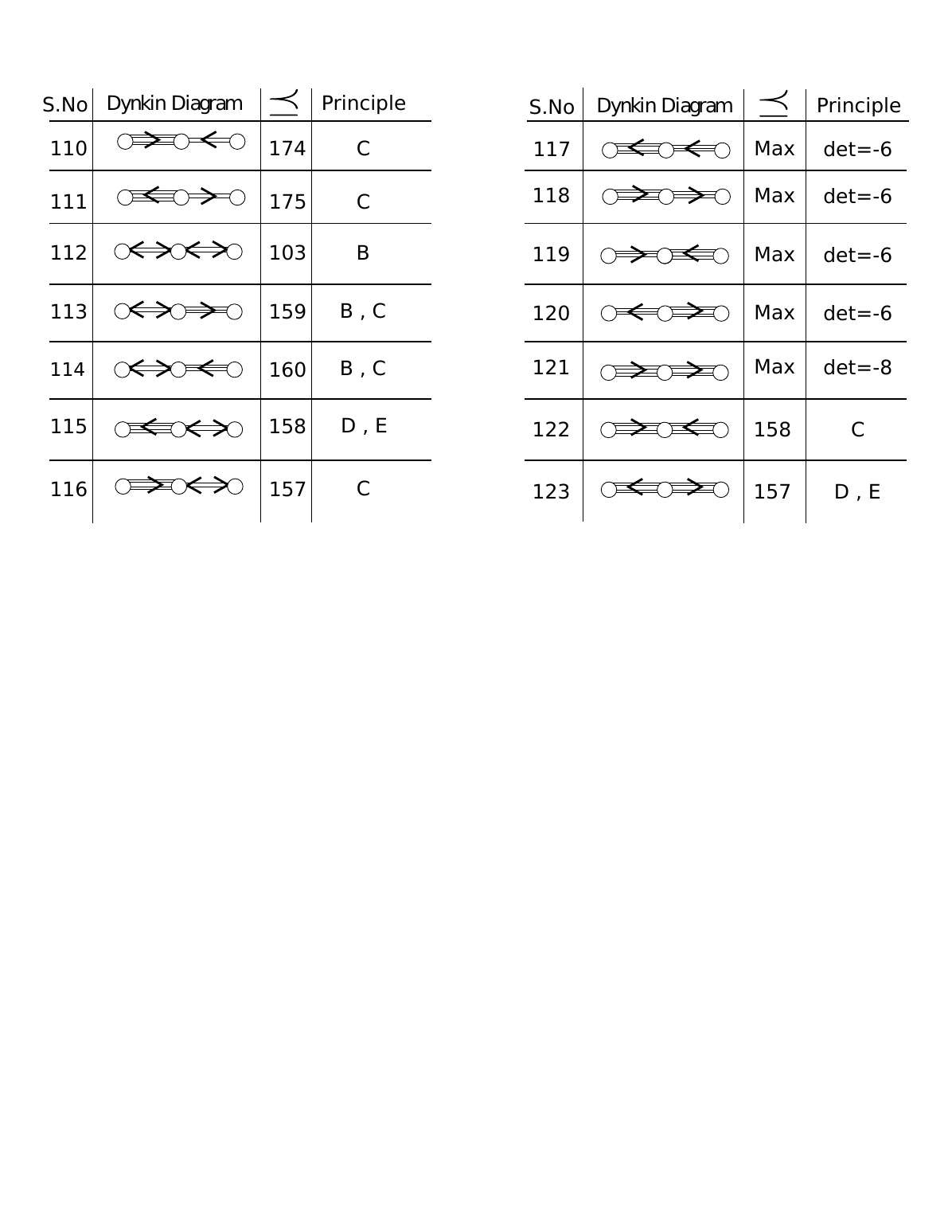}}};
\node at (0,-10) {\resizebox{4.25in}{4 in}{\includegraphics*[viewport=5 35 720 720 ]{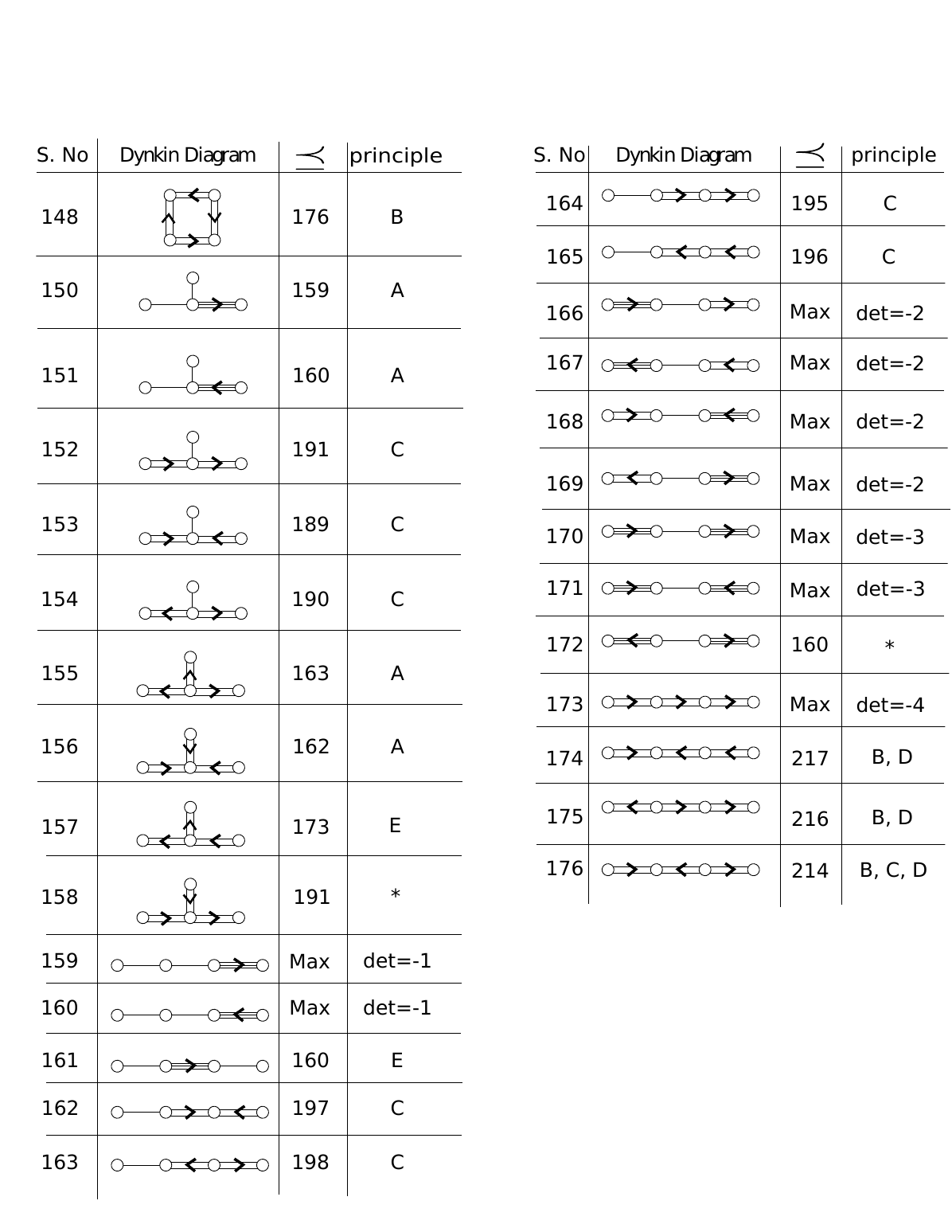}}};
\node at (9,-8.35) {\resizebox{4.25 in}{2 in }{\includegraphics*[viewport=5 275 720 720 ]{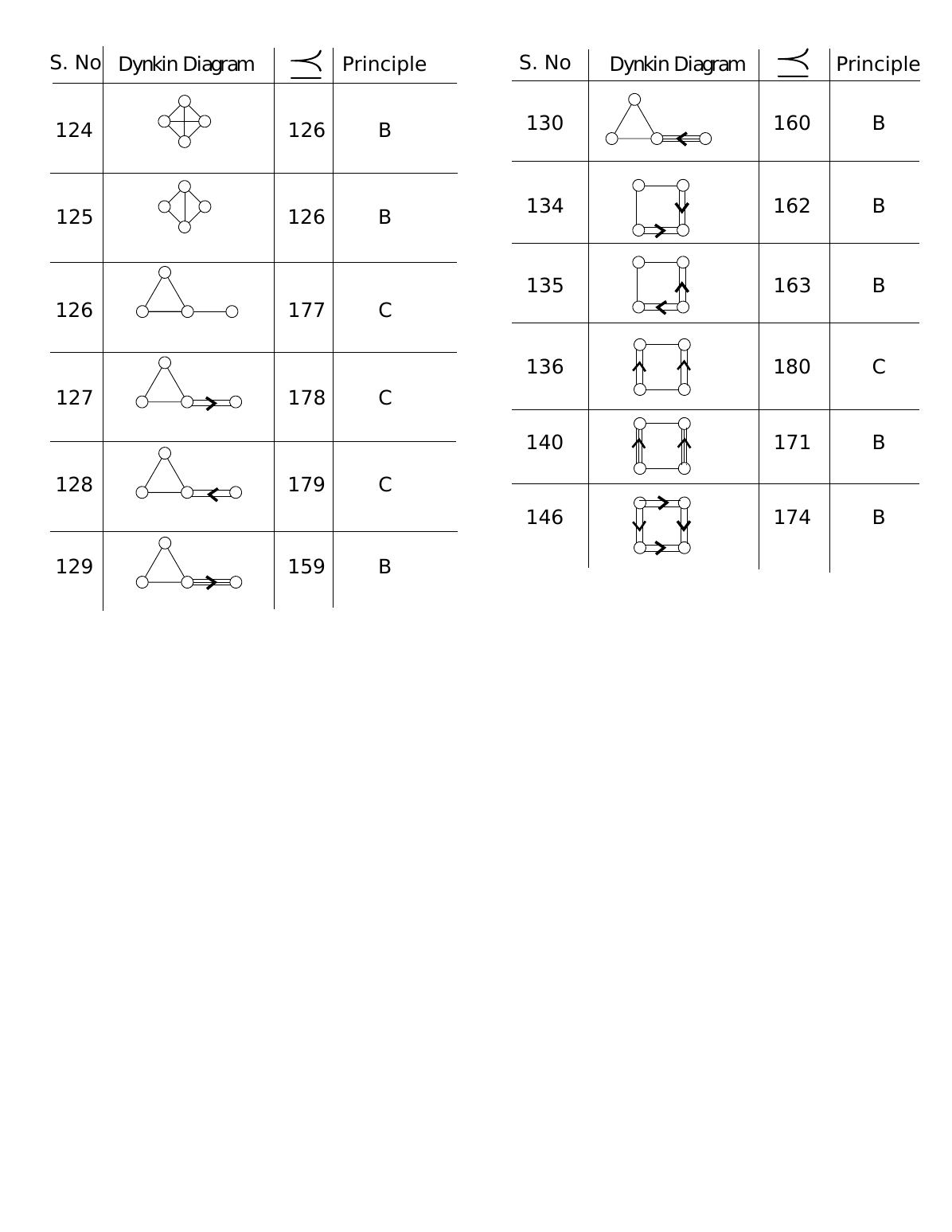}}};
\end{tikzpicture}

\newpage

\begin{tikzpicture}
\node at (3,29) {\text{TABLE II: Hyperbolic diagrams of ranks 5-8}};
  \node at (0,25) {\resizebox{4.25in}{4 in }{\includegraphics*[viewport=5 35 720 720 ]{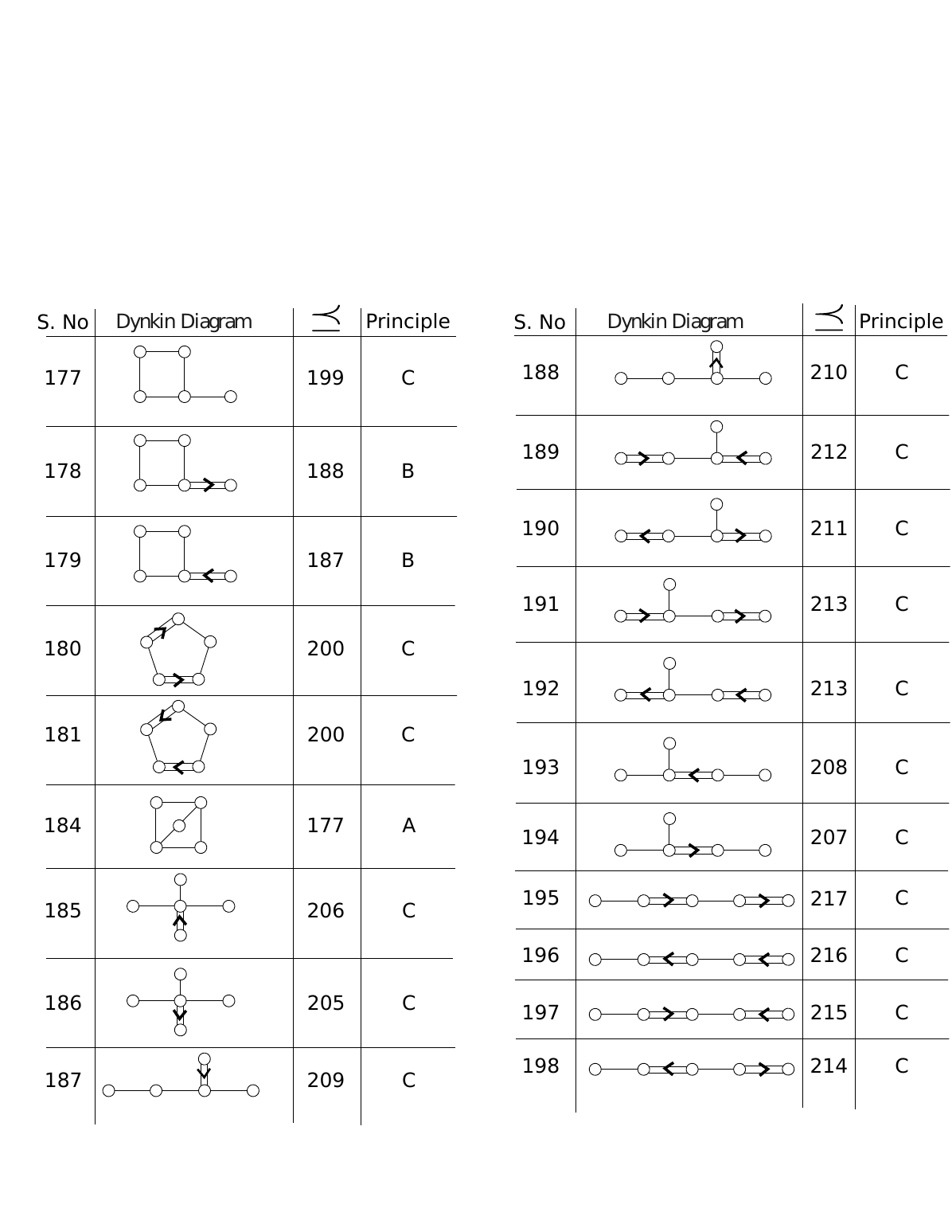}}};
  \node at (9,23.75) {\resizebox{4.25in}{3.75 in }{\includegraphics*[viewport=5 35 720 720 ]{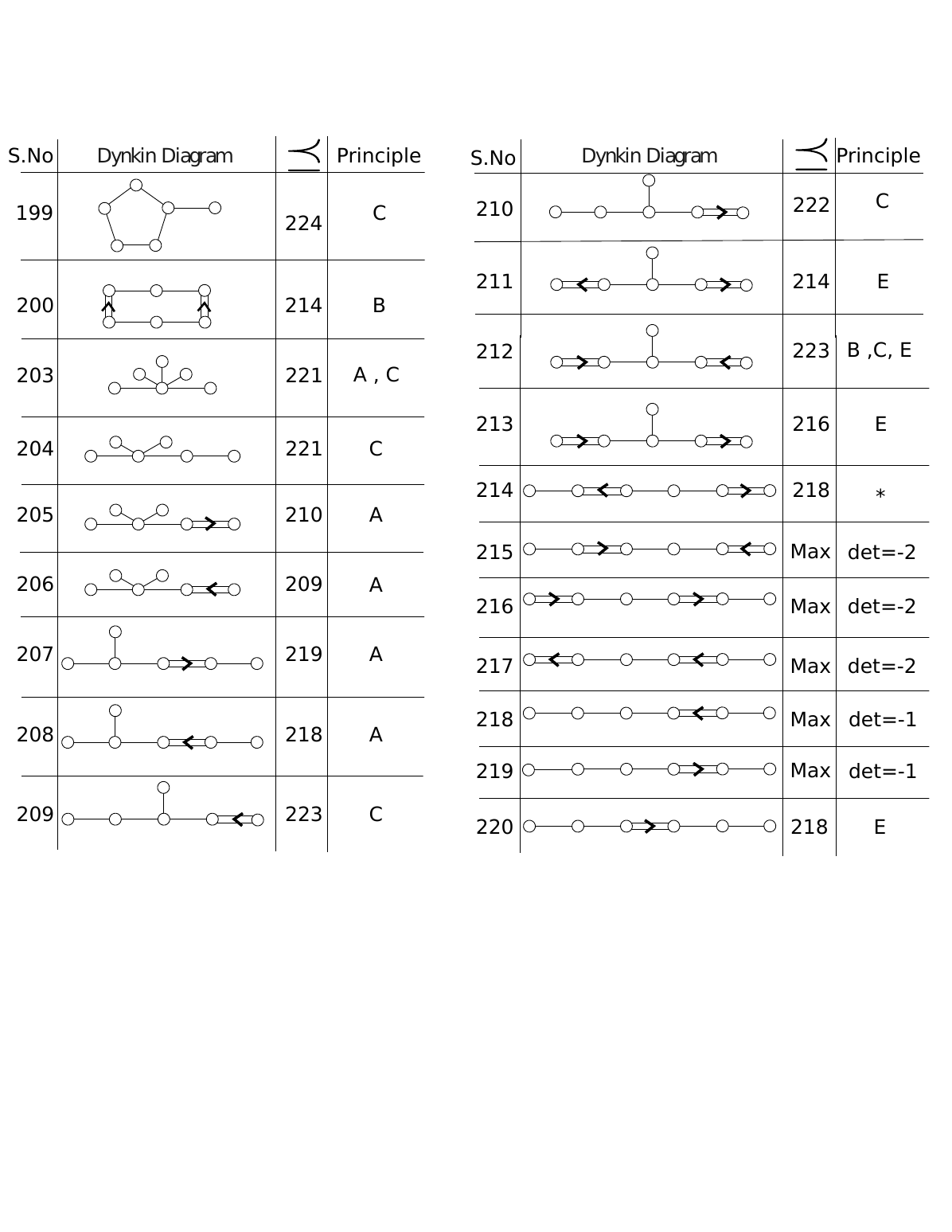}}};
\node at (0,18) {\resizebox{4.25in}{1.8in}{\includegraphics*[viewport=5 400 720 720 scale=0.05]{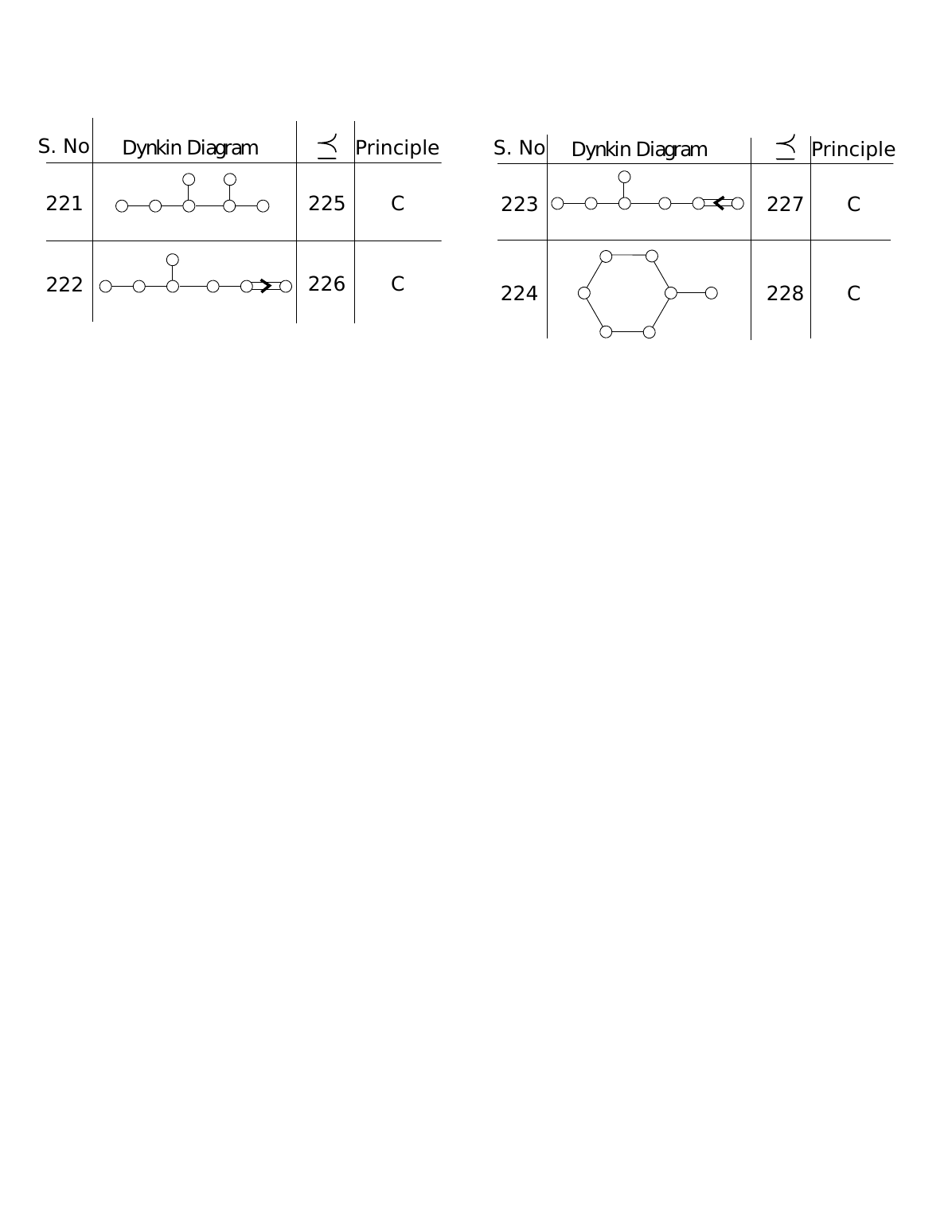}}};
\node at (9, 17.85) {\resizebox{4.25in}{1.65 in }{\includegraphics*[viewport=5 460 720 720 ]{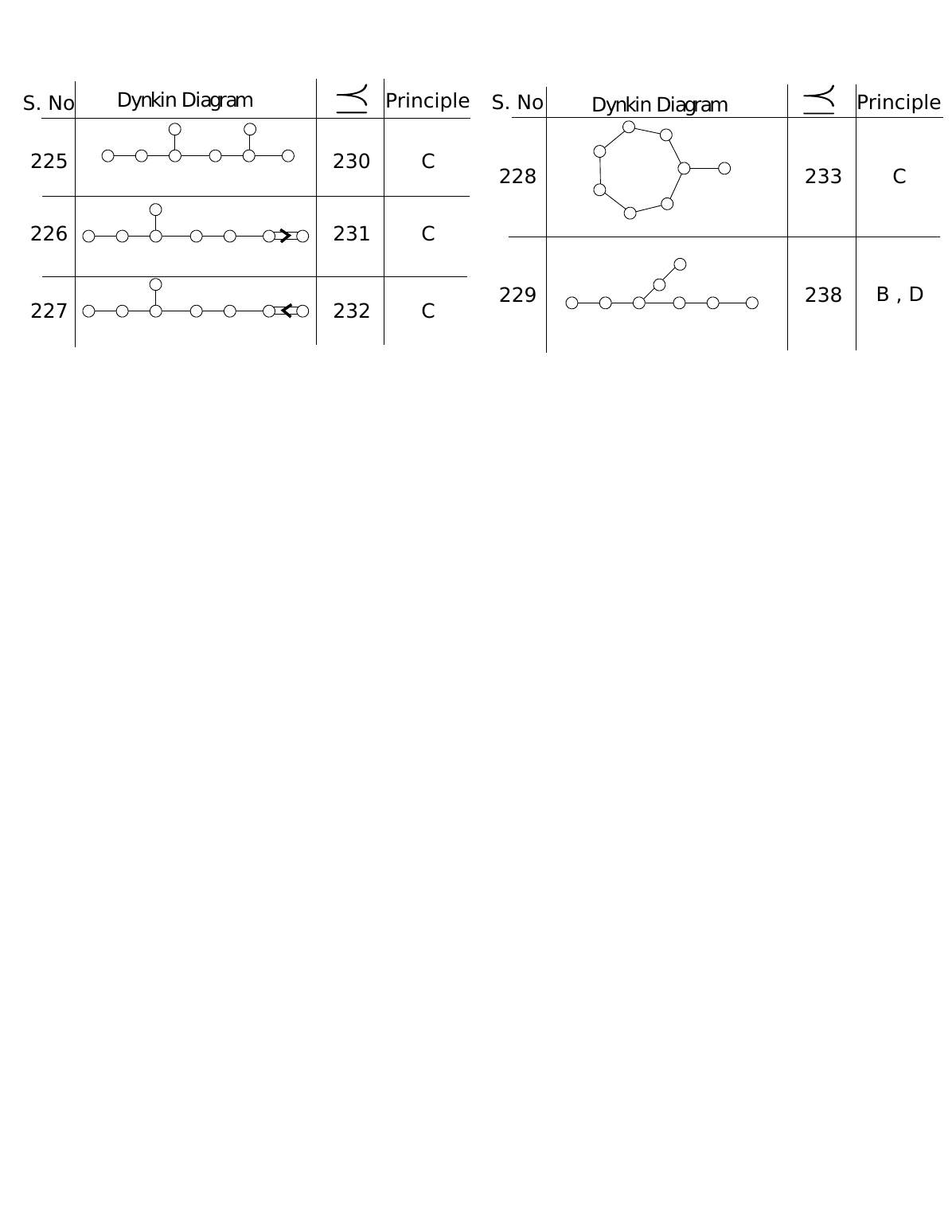}}};
\node at (3,16) {\text{Table III: Hyperbolic diagrams of ranks 9, 10}};
\node at (0,14) {\resizebox{4.25in}{1.65 in }{\includegraphics*[viewport=5 460 720 720 ]{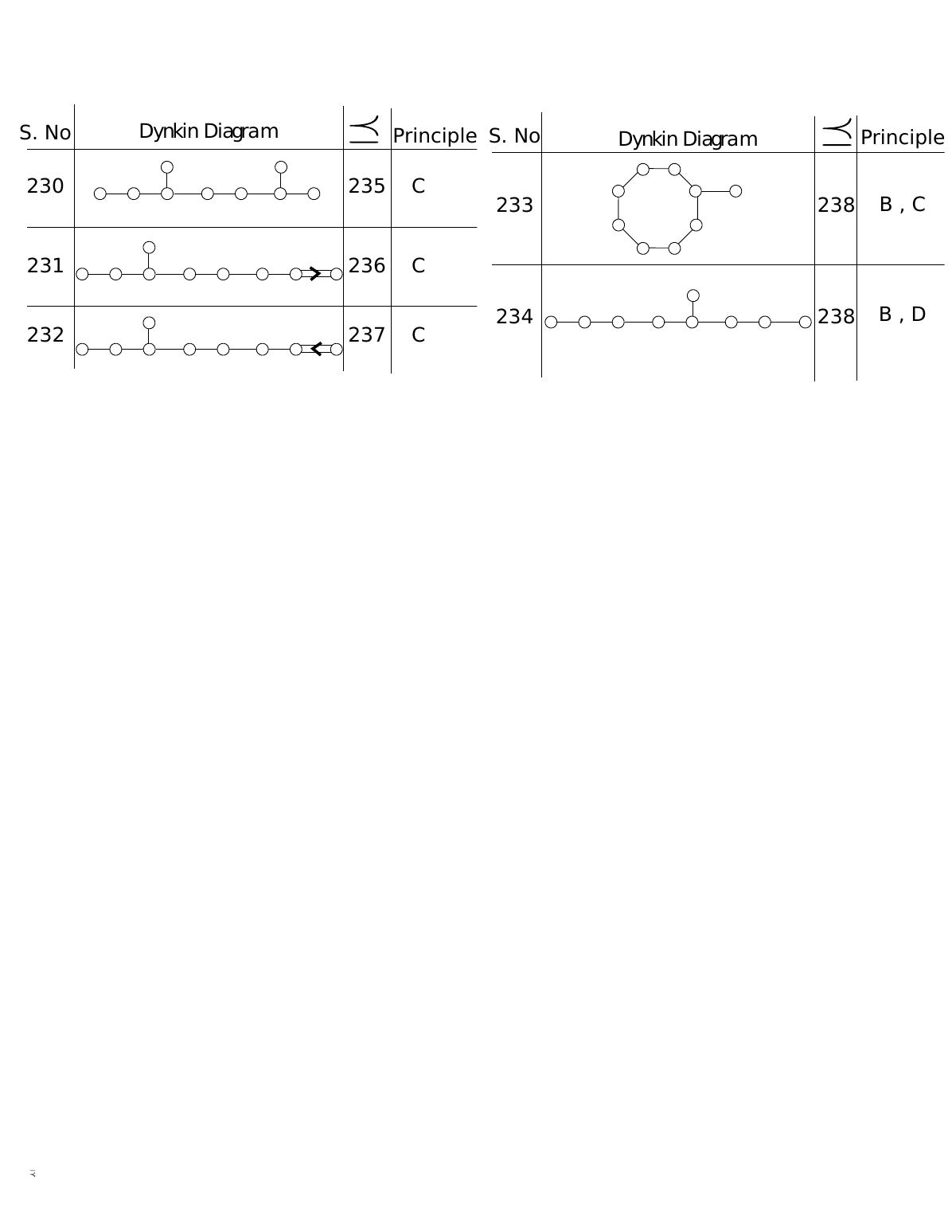}}};
\node at (9,14) {\resizebox{4.25in}{1.65 in }{\includegraphics*[viewport=5 400 720 720 ]{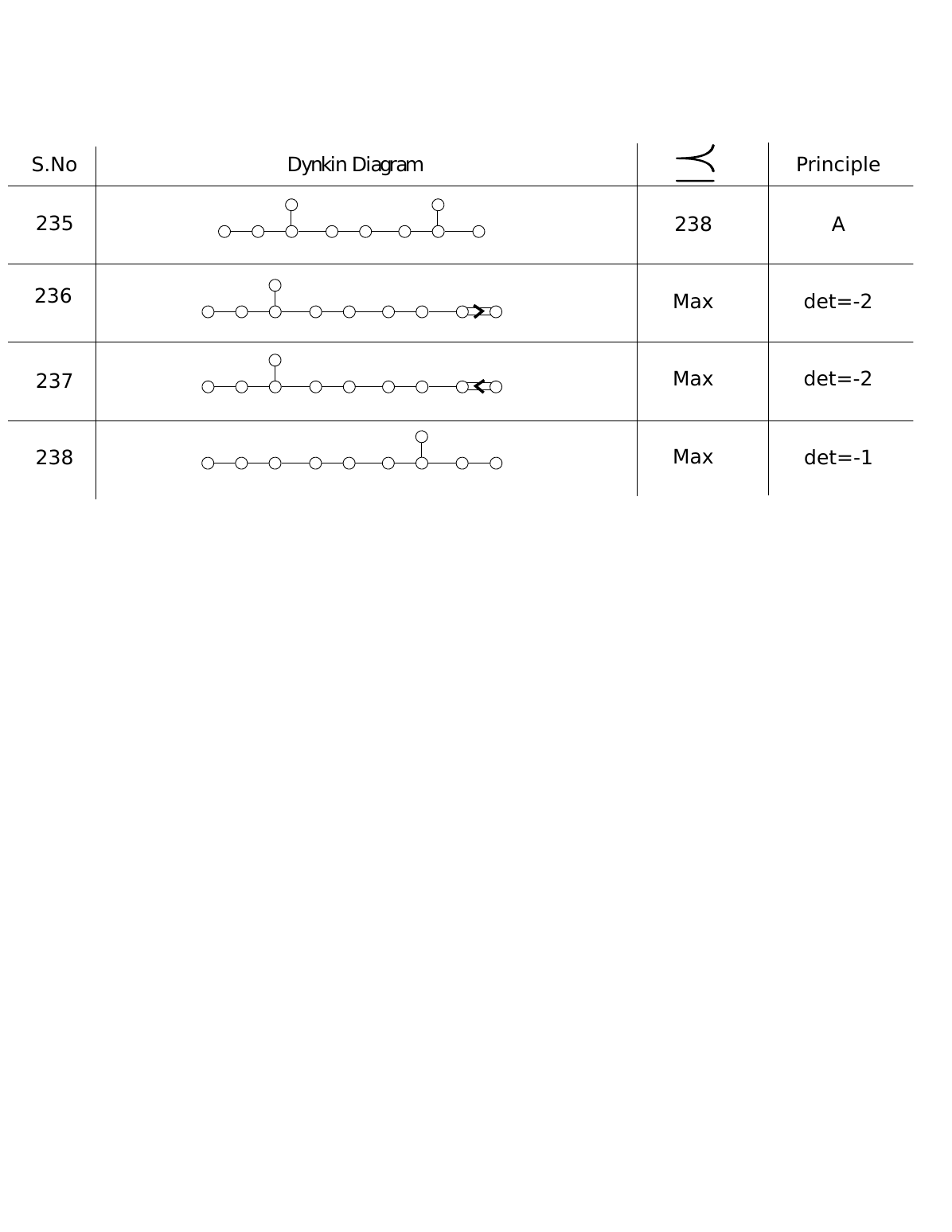}}};
\end{tikzpicture}

\bibliographystyle{plain}
\bibliography{pisystems}
\end{document}